\documentclass[onecolumn,fleqn,12pt,a4paper]{article}

\usepackage{amsmath,mathptmx,amssymb}
\usepackage{amssymb, amsthm, amsmath} 
\usepackage[dvips]{color}
\usepackage{graphicx}
\usepackage{latexsym}
\usepackage{comment}

\DeclareMathOperator*{\s-lim}{s-lim}
\newtheorem{Def}{Definition}
\newtheorem{Ass}[Def]{Assumption}
\newtheorem{Thm}[Def]{Theorem}
\newtheorem{Lem}[Def]{Lemma}

\newtheorem{Prop}[Def]{Proposition}

\theoremstyle{definition}
\newtheorem{Rmk}{Remark}

\newtheorem{Exm}[Rmk]{Example}

\newcommand{\maru}[1]{\underline{#1}}
\newcommand{\del}{\partial}

\title{Existence of wave operators with time-dependent modifiers for the Schr\"{o}dinger equations with long-range potentials on scattering manifolds}

\author{Shinichiro ITOZAKI\footnote{
Graduate School of Mathematical Sciences,
University of Tokyo.
3-8-1 Komaba, Meguro-ku
Tokyo, 153-8914
Japan.
E-mail: randy@ms.u-tokyo.ac.jp.
This work was partly supported by Grant-in-Aid for JSPS Fellows (No. 09J06551)}}

\date{January 4, 2012}

\begin{document}

\maketitle

\begin{abstract}
We construct time-dependent wave operators for Schr\"{o}dinger equations with long-range potentials 
on a manifold $M$ with asymptotically conic structure.
We use the two space scattering theory formalism,  
and a reference operator on a space of the form 
$\mathbb{R} \times \del M$,
where $\del M$
is the  boundary of $M$ at infinity.
We construct exact solutions to the Hamilton-Jacobi equation on the reference system $\mathbb{R} \times \del M$,
and prove the existence of the modified wave operators.
\end{abstract}

\section{Introduction}

In this paper,
we show the existence of wave operators 
for the Sch\"{o}dinger equations with long-range potentials on scattering manifolds, which have asymptotically conic structure at infinity (see Melrose \cite{M95} about scattering manifolds).
We employ the formulation of Ito-Nakamura~\cite{IN10}, which uses the two-space scattering framework of Kato \cite{Ka67}.
Following\\ H\"{o}rmander \cite{Ho76} and Derezi\'{n}ski and G\'{e}rard \cite{DG97}, we construct exact solutions to the Hamilton-Jacobi equation
and show the existence of the  modified two-space wave operators
using the stationary phase method.

Let $M$ be an $n$-dimensional smooth non-compact manifold  such that $M$ is decomposed to $M_{C}\cup M_{\infty}$,
where $M_{C}$ is relatively compact,
and $M_{\infty}$ is diffeomorphic to $\mathbb{R}_{+}\times \del M$ with a compact manifold $\del M$.
We fix an identification map:
\begin{equation*}
\iota : M_{\infty} \longrightarrow \mathbb{R}_{+} \times \partial M \ni (r, \theta).
\end{equation*}
We suppose $M_{c} \cap M_{\infty} \subset (0, 1/2)\times \partial M$ under this identification.
We also suppose that $\partial M$ is equipped with
a measure $H(\theta) d\theta$
where $H(\theta)$ is a smooth positive density
.

 Let $\{ \phi_{\lambda}: U_{\lambda} \to \mathbb{R}^{n-1} \} ,
U_{\lambda} \subset \partial M$,
be a local coordinate system of $\partial M$.
We set $\{ \tilde{\phi}_{\lambda} : \mathbb{R}_{+} \times U_{\lambda} \to \mathbb{R}\times\mathbb{R}^{n-1} \}$
to be a local coordinate system of $M_{\infty} \cong \mathbb{R}_{+} \times \partial M$,
and we denote $(r, \theta) \in \mathbb{R} \times \mathbb{R}^{n-1}$
to represent a point in $M_{\infty}$.
We suppose $G(x)$ is a smooth positive density on $M$ such that
\begin{equation*}
G(x)dx = r^{n-1}H(\theta)drd\theta \  \text{on} \ (\frac{1}{2}, \infty) \times \partial M \subset M_{\infty},
\end{equation*}
and we set
\begin{equation*}
\mathcal{H} = L^{2}(M, G(x)dx).
\end{equation*}
Let $P_{0}$ be a formally self-adjoint second order elliptic operator on $\mathcal{H}$ of the form:
\begin{equation*}
P_{0} = -\frac{1}{2}G^{-1}(\partial_{r}, \partial_{\theta}/r) G 
\begin{pmatrix}
1 + a_{1} 	& a_{2} \\
^{t}a_{2} 	& a_{3}
\end{pmatrix}
\begin{pmatrix}
\partial_{r}	\\
\partial_{\theta}/r
\end{pmatrix}
\ \ \text{on} \ \tilde{M}_{\infty} = (1, \infty) \times \del M
\end{equation*}
where $a_{1}, a_{2},$ and $ a_{3}$ are real-valued smooth tensors.

\begin{Ass}
For any $l \in \mathbb{Z}_{+}$, $\alpha \in \mathbb{Z}_{+}^{n-1}$, there is $C_{l, \alpha}$ such that
\begin{gather*}
|\partial_{r}^{l}\partial_{\theta}^{\alpha}a_{j}(r, \theta ) | \leq C_{l, \alpha} r^{-\mu_{j}-l}			
\end{gather*}
on $\tilde{M}_{\infty}$, where $\mu_{j} \geq 0$.
Note that we use the coordinate system in $M_{\infty}$ described above.
\end{Ass}

We will construct a time-dependent scattering theory for $P_{0}+V$ on $\mathcal{H}$ where
$V$ is a potential.
\begin{Def}
Let $\mu_{s} >0$.
A finite rank differential operator $V^{S}$ of the form 
$V^{S} = \sum_{l, \alpha} V_{l, \alpha}^{S}(r,\theta) \partial_{r}^{l}\partial_{\theta}^{\alpha}$  on  $M_{\infty}$
is said to be a short range perturbation of $\mu_{S}$ type
if for every $l, \alpha$ the coefficient $V_{l, \alpha}^{S}$ is a $L^{2}_{loc}$ tensor and satisfies
\begin{equation*}
\int_{\mathbb{R}_{+} \times U_{\lambda}}
|V_{l, \alpha}^{S}(x)|^{2}\langle r \rangle^{-M} G(x)dx	 < \infty \\
\end{equation*}
for some $M$, and almost every $(\rho_{0}, \theta_{0}) \in \mathbb{R}\times \partial M $ has 
a neighborhood $\omega_{\rho_{0}, \theta_{0}}$
 such that
\begin{equation*}
\int_{1}^{\infty}(\int_{(\rho, \theta) \in \omega_{\rho_{0}, \theta_{0}}}|V_{l, \alpha}^{S}(t\rho, \theta)|^{2}
d\rho H(\theta ) d\theta ) ^{1/2} t^{\mu_{S} |\alpha|}dt < \infty .
\end{equation*}

Let
$\mu_{L} > 0$.
$V^{L}$ is called a long-range smooth potential if $V^{L}$ is a real-valued $C^{\infty}$ function with support in $\tilde{M}_{\infty}$,
and satisfies for any indices $l, \alpha$, 
\begin{equation*}
|D_{r}^{j}D_{\theta}^{\alpha}V^{L}(r, \theta )|\leq C_{j, \alpha} r^{-\mu_{L} -j}.
\end{equation*}

A differential operator $V$ on $M$ is called an admissible long-range perturbation of $P_{0}$ if 
$V$ is of the form $V = V^{S} + V^{L}$ 
where $V^{S}$ is a short range perturbation of $\mu_{S}$ type and $V^{L}$ is a long-range smooth potential
and 
\begin{gather*}
\epsilon = \mu_{1} = \mu_{2} = \mu_{L} > 0,\ \ 
\mu_{3} = 0,\ \ 
\mu_{S} = 1- \epsilon .
\end{gather*}
\end{Def}

\begin{Exm}
If $V^{S} = V^{S}(r, \theta)$ is a multiplication operator and $|V^{S}(r, \theta)| \leq C r^{-1-\eta}$, $\eta > 0$, then $V^{S}$ satisfies the short-range condition above.

If $V^{S} = \sum_{|\alpha| = 1}V_{\alpha}^{S}\del_{\theta}^{\alpha}$ and $|V_{\alpha}^{S}(r, \theta)| \leq C r^{-1-\mu_{S} -\eta}, \eta > 0$, then $V^{S}$ satisfies the short-range condition above. As the order of the derivative with respect to $\theta$-variable increases, we need more rapid decay conditions on the coefficients.
\end{Exm}

\begin{Rmk}
If $V^{S}$ is a smooth function, then $P_{0} + V$ is essentially self-adjoint.
More generally, if $V^{S}$ is at most second-order differential operator with ``small'' smooth coefficients, then $V^{S}$ is $P_{0}$-bounded with relateve bound less than one, and $P_{0} + V$ is essentially self-adjoint.
In this paper, we assume that $P_{0} + V$ is essentially self-adjoint on suitable domains (see Theorem \ref{Wave_MainThm}) and do not investigate the conditions of self-adjointness.
\end{Rmk}

\begin{Rmk}
If we assume $\del M$ is equipped with a positive (2, 0)-tensor $h = (h^{jk}(\theta))$, for some $\epsilon > 0$,
\begin{gather*}
|\del_{r}^{l} \del_{\theta}^{\alpha}(a_{3}(r, \theta) - h(\theta) ) | \leq C_{l, \alpha} r^{ - \epsilon - l},
\end{gather*} 
and $V^{S} = 0$, then $P_{0} + V $ has a self-adjoint extension $H$ and corresponds (via a unitary equivalence) to the Laplacian on Riemannian manifolds with asymptotically conic structure. Since $\epsilon > 0$, our model includes the scattering metric of long-range type described in \cite{IN09}. Thus our results are generalizations of \cite{IN10}.
\end{Rmk}

We prepare a reference system as follows:
\begin{gather*}
M_{f} = \mathbb{R}\times \partial M,\ \  \mathcal{H}_{f} = L^{2}(M_{f}, H(\theta)dr d\theta ),\ \ 
P_{f} = -\frac{1}{2}\frac{\partial^{2}}{\partial r^{2}} \ \text{on} \ M_{f}
\end{gather*}
Note that $P_{f}$ is essentially self-adjoint on $C_{0}^{\infty}(M_{f})$,
and we denote the unique self-adjoint extension by the same symbol.
Let $j(r) \in C^{\infty}(\mathbb{R})$ be a real-valued function such that
$j(r) = 1$ if $r \geq 1$ and $j(r)=0$ if $r\leq 1/2$.
We define the identification operator $J: \mathcal{H}_{f} \to \mathcal{H}$ by
\begin{equation*}
(Ju)(r,\theta)= r^{-(n-1)/2}j(r)u(r, \theta) \ \ \text{if}\ \  (r,\theta) \in M_{\infty}
\end{equation*}
and $Ju(x) = 0$ if $x \notin M_{\infty}$, where $u\in \mathcal{H}_{f}$.
We denote the Fourier transform with respect to $r$-variable by $\mathcal{F}$:
\begin{equation*}
\mathcal{F}u(\rho,\theta) = \int_{-\infty}^{\infty}e^{-ir\rho}u(r, \theta)dr,
\ \  \text{for}\ \  u \in C_{0}^{\infty}(M_{f}).
\end{equation*}
We decompose the reference Hilbert space $\mathcal{H}_{f}$ as 
$\mathcal{H}_{f} = \mathcal{H}_{f}^{+}\oplus \mathcal{H}_{f}^{-}$,
where $\mathcal{H}_{f}^{\pm}$ are defined by
\begin{gather*}
\mathcal{H}_{f}^{+} \ = \ \{ u \in \mathcal{H} | \text{supp}(\mathcal{F}u) \subset [0, \infty) \times \partial M\},		\\
\mathcal{H}_{f}^{-} \ = \ \{ u \in \mathcal{H} | \text{supp}(\mathcal{F}u) \subset (-\infty, 0] \times \partial M\}.
\end{gather*}
We use the following notation throughout the paper:
For $x \in M$, we write
\begin{equation*}
\langle x \rangle = \langle r \rangle = 
\begin{cases}
1 + rj(r)  	&\text{for} \ x \in M_{\infty},\\
1		&\text{for} \ x \in M_{c}	.
\end{cases}
\end{equation*}

We state our main theorem.

\begin{Thm}\label{Wave_MainThm}
Let $V = V^{L}+V^{S}$ be an admissible long-range perturbation of $P_{0}$,
and $V$ is symmetric on $J\mathcal{F}^{-1}C_{0}^{\infty}(M_{f})$,
and  $P_{0} + V$ has a self adjoint extension $H$.
Let $S(t, \rho , \theta)$ be a solution to the Hamilton-Jacobi equation which is constructed in Theorem  \ref{SThmGlobal}.
Then the modified wave operators
\begin{equation*}
\Omega_{\pm} = \s-lim_{t\to \pm \infty}e^{itH} J e^{-iS(t, D_{r}, \theta )}
\end{equation*}
exist, and are partial isometries from $\mathcal{H}_{f}^{\pm}$ into $\mathcal{H}$
intertwining $H$ and $P_{f}$:
\begin{equation*}
e^{i s H}\Omega_{\pm} = \Omega_{\pm} e^{i s P_{f}}.
\end{equation*}
\end{Thm}

We refer Reed and Simon \cite{RS72-80}, Derezi\'{n}ski and G\'{e}rard \cite{DG97}, and Yafaev \cite{Ya00} for general concepts of wave operators and scattering theory for Sch\"{o}dinger equations. We here briefly review the history of wave operators.
The concept of wave operator was introduced by M\o ller \cite{Mo45}.
The existence of wave operators has long been studied (see Cook \cite{Co57} and Kuroda \cite{Ku59}) for short range potentials, which decay faster than the Coulomb potential.
For the Coulomb potential, it was proved by Dollard \cite{Do64, Do71} that the wave operators do not exist unless the definition is modified. Dollard introduced the concept of the modified wave operators $\s-lim_{t \to \pm \infty} e^{itH} e^{-iS(t, D_{x})}$. Buslaev-Mateev \cite{BM70} showed the existence of modified wave operators by using stationary phase method and  by employing an approximate solution to the Hamilton-Jacobi equation as a modifier function $S(t, \xi)$. H\"{o}rmander \cite{Ho76} constructed exact solutions to the Hamilton-Jacobi equation (see also \cite{Ho83-85} vol. IV). 

The spectral properties of Laplace operators on a class of non-compact manifolds were studied by Froese, Hislop and Perry \cite{FH89, FHP91}, and Donnelly \cite{D99} using the Mourre theory (see, the original paper Mourre \cite{Mo81}, and Perry, Sigal, and Simon \cite{PSS81}).
In early 1990s, Melrose introduced a new framework of scattering theory on a class of Riemannian manifolds with metrics called scattering metrics (see \cite{M95} and references therein), and showed that the absolute scattering matrix, which is defined through the asymptotic expansion of generalized eigenfunctions, is a Fourier integral operator. Vasy \cite{Va98} studied Laplace operators on such manifolds with long-range potentials of Coulomb type decay ( $|V(r, \theta)| \leq C r^{-1}$).

Ito and Nakamura \cite{IN10} studied a time-dependent scattering theory for \\
Schr\"{o}dinger operators on scattering manifolds. They used  the two-space scattering framework of Kato \cite{Ka67} with a simple reference operator $D_{r}^{2}/2$ on a space of the form $\mathbb{R} \times \del M$, where $\del M$ is the boundary of the scattering manifold $M$.

We employ the formulation of Ito and Nakamura \cite{IN10}, and consider general long-range metric perturbations and potential perturbations. We assume that  the scalar potential decay as $|V(r, \theta)| \leq C r^{- \epsilon}, \epsilon > 0$.

We make some remarks along with the outline of the proof.
The time-\\dependent modifier function $S(t, \rho, \theta)$ is not uniquely determined.
Our choice is a solution to the Hamilton-Jacobi equation on the reference manifold $\mathbb{R} \times \del M$
with the long-range potential $V^{L}$:
\begin{gather}
h(\frac{\del S}{\del \rho}, \theta, \rho, -\frac{\del S}{\del \theta}) = \frac{\del S}{\del t} ,\label{Introduction_Hamilton_Jacobi_eq} \\
h(r, \theta, \rho, \omega) = \frac{1}{2}\rho^{2} + 
\frac{1}{2}a_{1}\rho^{2} + \frac{1}{r}a_{2}^{j}\rho \omega_{j} + \frac{1}{r^{2}}a_{3}^{j k }\omega_{j}\omega_{k} 
+V^{L}, \notag
\end{gather}
for large $t$ and for every $\rho$ in any fixed compact set of $\mathbb{R} \setminus \{ 0 \}$, where $h$ is the corresponding classical Hamiltonian.
We choose $\rho$ and $\theta$ as variables of $S$ because $\rho$ and $\theta$ components of the classical trajectories have limits as $t$ goes to infinity. The time-dependent modifier $e^{-S(t, D_{r}, \theta)}$ is a Fourier multiplier in $r$-variable for each $\theta$ and we only need to consider the $1$-dimensional Fourier transform with respect to $r$-variable.
We construct solutions to the Hamilton-Jacobi equation mainly following J. Derezi\'{n}ski and C. G\'{e}rard \cite{DG97}.

In Section \ref{Classical trajectories with slowly-decaying time-dependent force}, we consider the boundary value problem for Newton equation on $\mathbb{R} \times \del M$ with time-dependent slowly-decaying forces, which decay in time (Definition  \ref{slowly-decaying_force}). 
In Theorem  \ref{E4Neq}, we construct solutions and show several estimates. 
We use an integral equation and Banach's contraction mapping theorem (Proposition  \ref{P:contraction}, refer Derezi\'{n}ski \cite{De91} and Section 1.5 of \cite{DG97}). 
In the definition of slowly-decaying forces (Definition \ref{slowly-decaying_force}) and the function spaces (Definition \ref{Banach_Space_of_functions_Newton}), we assume different decaying rates on different variables $r, \theta, \rho,$ and $\omega$. These are efficiently used to show Proposition \ref{P:contraction}.
We observe that the classical trajectories will stay in outgoing (incoming ) regions as $t \to +\infty ( - \infty)$.

In Section \ref{Classical trajectories with long-range time-independent force}, we consider  Newton equations with time-independent long-range forces￼ which decay in space (Definition  \ref{long-range_force}) in appropriate outgoing (incoming) regions. 
By inserting time-dependent cut-off functions, we introduce an effective time-dependent force and reduce the time-independent problem to the time-dependent one (Theorem  \ref{CTini-fi}).
Our model (the Hamiltonian flow induced by the classical Hamiltonian) turns out to fit into this framework (Lemma  \ref{HtoF}).
These tricks are also used in \cite{DG97} for Hamiltonians with long-range potentials on Euclidean spaces.

Finally, in Section \ref{Solutions of the Hamilton-Jacobi equation}, in Theorem \ref{SThmGlobal} we construct exact solutions to the Hamilton-Jacobi equation, using the classical trajectories with their dependence on initial data. Here we use the idea by H\"{o}rmandor \cite{Ho76}, see also Section 2.7 of \cite{DG97}.
We show that these solutions with their derivatives satisfy ``good estimates'', which are used to show the existence of the modifiers. Once we obtain a suitable modifier $S(t, \rho, \theta)$, we can show the existence of modified wave operators through stationary phase method (Section  \ref{Sec_Proof_of_the_main_Thm}).

Using the Cook-Kuroda method (see Cook \cite{Co57}, and Kuroda \cite{Ku59}) and \\$1$-dimensional Fourier transform, we deduce the proof of the main theorem to estimates of  the integral (Proposition \ref{LRProp}):
\begin{gather*}
\int [h(r, \theta, \rho, -\frac{\del S}{\del \theta}(t, \rho, \theta))
 - h(\frac{\del S}{\del \rho}(t, \rho, \theta), \theta, \rho, -\frac{\del S}{\del \theta}(t, \rho, \theta))]\\
\cdot e^{i r \rho -iS(t, \rho, \theta)} \hat{u}(\rho, \theta) d\rho. \notag
\end{gather*}
In Section \ref{Sec_Proof_of_the_main_Thm}, we apply the stationary phase method (H\"{o}rmander~\cite{Ho83-85} Section 7.7). In the asymptotic expansion of the above integral, the terms in which $h$ is not differentiated vanish since the equation $r = \del S/\del \rho$ holds at the stationary points. To show the uniformly boundedness of constants which appear in the asymptotic expansions of the integral, we construct diffeomorphisms in small neighborhoods of the stationary points which transform the phase function into quadratic forms there (Lemma \ref{psiLem}). In the constructions of these diffeomorphisms, we use the estimates on the modifier function $S$.

\textbf{Notations}
We use the following notation throughout the paper.
Let $t\in \mathbb{R}$ and $s$ be a parameter.
We write $f(t, s) \in  g(s) \mathrm{O}(\langle t \rangle^{-m})$ if $f(t,s) \leq C g(s) \langle t \rangle^{-m}$ uniformly for $t$ and $s$.
We denote $f(t, s) \in g(s)\mathrm{o}(t^{0})$ if 
$\lim_{t \to \infty} f(t, s)/g(s) = 0$.

\textbf{ACKNOWLEDGMENTS}

I would like to express my thank to my advisor, Professor Shu Nakamura
for providing me various supports, comments, and encouragements.

\section{Classical mechanics}
In this section, we study classical trajectories and solutions to the Hamilton-Jacobi equation.

\subsection{Classical trajectories with slowly-decaying time-dependent force}\label{Classical trajectories with slowly-decaying time-dependent force}
Let $(r, \theta, \rho, \omega) \in T^{*}(\mathbb{R} \times \mathbb{R}^{n-1})$
and consider the Newton's equation:
\begin{gather}
(\dot{r}, \dot{\theta}, \dot{\rho}, \dot{\omega})(t) = (\rho + F_{r}, F_{\theta}, F_{\rho}, F_{\omega})(t, (r, \theta, \rho, \omega)(t) ) \label{Newton}
\end{gather}
where
\begin{gather*}
F=F(t, r, \theta, \rho, \omega) = (F_{r}, F_{\theta}, F_{\rho}, F_{\omega})(t, r, \theta, \rho, \omega)
\end{gather*}
is a time-dependent force.
Let $\epsilon >0$
and $\tilde{\epsilon} = \frac{1}{2}\epsilon$.

\begin{Def}\label{slowly-decaying_force}
A time-dependent force $F$ is said to be slow-decaying if $F$ satisfies
\begin{gather}
\sup_{(r, \rho, \theta, \omega) \in T^{*}(\mathbb{R} \times \mathbb{R}^{n-1})} |\del_{r}^{l} \del_{\theta}^{\alpha} \del_{\rho}^{k} \del_{\omega}^{\beta}(F_{r}, F_{\theta}, F_{\rho}, F_{\omega})(t)|
\in \mathrm{O}( \langle t \rangle^{-n_{r, \theta ,\rho ,\omega }(l, \alpha, k, \beta) })
\label{SDFAss}
\end{gather}
where
\begin{gather}
n_{r}(l, \alpha, k, \beta) = m(l, \alpha, k+1, \beta),	\ \ 	
n_{\theta}(l, \alpha, k, \beta) = m(l, \alpha, k, \beta +e_{i}),		\notag \\
n_{\rho}(l, \alpha, k, \beta) = m(l+1, \alpha, k, \beta),		\ \ 
n_{\omega}(l, \alpha, k, \beta) = m(l, \alpha +e_{i}, k, \beta),	\notag \\
m(l, \alpha, 0, 0) = l + \epsilon, \ \ 
 m(l, \alpha, 1, 0) = l + \epsilon,  \ \ 
m(l, \alpha, 2, 0) = l + \epsilon,    \label{index}\\
m(l, \alpha, 0, e_{i}) = l + 1 + \tilde{\epsilon}, \ \ 
m(l, \alpha, 1, e_{i}) = l + 1+ \epsilon,  \notag \ \ 
m(l, \alpha, 0, e_{i} + e_{j}) = l + 2,  \notag \\
m(l, \alpha, k, \beta) = +\infty, \ \ \ \ \ \text{if}\ \ \ \ \  k +|\beta| \geq 3, \notag 
\end{gather}
$i, j = 1, \cdots ,n-1$, and $e_{i} = (0, \cdots , 1,0, \cdots,  0) \in \mathbb{Z}_{+}^{n-1}$ is the canonical unit vector, i.e.,  every component of $e_{i}$ is $0$ except $i$-th component.
\end{Def}

In the next theorem, we show the unique existence of trajectories for the dynamics  \eqref{Newton} where the boundary conditions are the initial position and the final momentum.

\begin{Thm}\label{E4Neq}
Assume that $F$ is a time-dependent slowly-decaying force in the sense of Definition  \ref{slowly-decaying_force}.
Then there exists $T$ such that 
if $T \leq t_{1} < t_{2} \leq \infty$
and 
$(r_{i}, \theta_{f}, \rho_{f}, \omega_{i}) \in T^{*}(\mathbb{R} \times \mathbb{R}^{n-1})$,
there exists a unique trajectory
\begin{gather*}
[t_{1}, t_{2}] \ni s \mapsto (\tilde{r}, \tilde{\theta}, \tilde{\rho}, \tilde{\omega})(s, t_{1}, t_{2}, r_{i}, \theta_{f}, \rho_{f}, \omega_{i})
\end{gather*}
satisfying
\begin{gather*}
\del_{s}(\tilde{r}, \tilde{\theta}, \tilde{\rho}, \tilde{\omega})(s, t_{1}, t_{2}, r_{i}, \theta_{f}, \rho_{f}, \omega_{i}) \\
= (\rho + F_{r}, F_{\theta}, F_{\rho}, F_{\omega})(s, (\tilde{r}, \tilde{\theta}, \tilde{\rho}, \tilde{\omega})(s, t_{1}, t_{2}, r_{i}, \theta_{f}, \rho_{f}, \omega_{i}) ),\\
(\tilde{r}, \tilde{\omega})(t_{1}, t_{1}, t_{2}, r_{i}, \theta_{f}, \rho_{f}, \omega_{i})
= (r_{i}, \omega_{i}), \ \ 
(\tilde{\theta}, \tilde{\rho})(s, t_{1}, t_{2}, r_{i}, \theta_{f}, \rho_{f}, \omega_{i})
= (\theta_{f}, \rho_{f})    .
\end{gather*}
We set $\maru{r}(s), \maru{\theta}(s), \maru{\rho}(s), \maru{\omega}(s)$ by
\begin{gather*}
\maru{r}(s)		=\tilde{r}(s) - r_{i} - (s-t_{1}) \rho_{f}, \ \ 
\maru{\theta}(s)=\tilde{\theta}(s) - \theta_{f},\\
\maru{\rho}(s)	=\tilde{\rho}(s) - \rho_{f}, \ \ 
\maru{\omega}(s)=\tilde{\omega}(s) - \omega_{i}.
\end{gather*}
Moreover the solution satisfies the following estimates uniformly for \\
$T \leq t_{1} \leq s \leq t_{2} \leq \infty$, 
$(r_{i}, \theta_{f}, \rho_{f}, \omega_{i}) \in T^{*}(\mathbb{R} \times \mathbb{R}^{n-1})$ :
\begin{gather}
\begin{array}{l}
|\maru{r}(s)| \in \mathrm{o}(s^{0}) |s-t_{1}|,  	\ \
|\maru{\theta}(s)| \in \mathrm{O}(s^{-\tilde{\epsilon}}),			\\
|\maru{\rho}(s)| \in \mathrm{O}(s^{-\tilde{\epsilon}}),			\ \ 
|\maru{\omega}(s)| \in \mathrm{o}(s^{0}) |s-t_{1}|^{1-\tilde{\epsilon}} ,
\end{array}
\label{0thderivative}\\
\left(
\begin{array}{cccc}
\del_{r_{i}}\maru{r}(s)			&
\del_{\theta_{f}}\maru{r}(s)	&
\del_{\rho_{f}}\maru{r}(s)		&
\del_{\omega_{i}}\maru{r}(s)	\\
\del_{r_{i}}\maru{\theta}(s)			&
\del_{\theta_{f}}\maru{\theta}(s)	&
\del_{\rho_{f}}\maru{\theta}(s)		&
\del_{\omega_{i}}\maru{\theta}(s)	\\
\del_{r_{i}}\maru{\rho}(s)			&
\del_{\theta_{f}}\maru{\rho}(s)	&
\del_{\rho_{f}}\maru{\rho}(s)		&
\del_{\omega_{i}}\maru{\rho}(s)	\\
\del_{r_{i}}\maru{\omega}(s)			&
\del_{\theta_{f}}\maru{\omega}(s)	&
\del_{\rho_{f}}\maru{\omega}(s)		&
\del_{\omega_{i}}\maru{\omega}(s)	
\end{array}
\right)			\notag \\
\in
\left(
\begin{array}{c}
\mathrm{o}(s^{0}) |s-t_{1}|		\\
\mathrm{O}(1)								\\
\mathrm{O}(s^{-\tilde{\epsilon}})			\\
\mathrm{o}(s^{0}) |s-t_{1}|^{1-\tilde{\epsilon}}
\end{array}
\right) \otimes
(t_{1}^{-1-\tilde{\epsilon}},
t_{1}^{- \tilde{\epsilon}},
t_{1}^{-\tilde{\epsilon}},
t_{1}^{-1}
),
\label{1stderivative}\\
\del_{r_{i}}^{l} \del_{\theta}^{\alpha} \del_{\rho}^{k} \del_{\omega}^{\beta}
\left(
\begin{array}{c}
\maru{r}\\
\maru{\theta}\\
\maru{\rho}\\
\maru{\omega}
\end{array}
\right)
\in
\left(
\begin{array}{c}
\mathrm{o}(s^{0}) |s-t_{1}|		\\
\mathrm{O}(1)								\\
\mathrm{O}(s^{-\tilde{\epsilon}})			\\
\mathrm{o}(s^{0}) |s-t_{1}|^{1-\tilde{\epsilon}}
\end{array}
\right) \cdot
t_{1}^{-l-|\beta |}.
\label{nthderivative}
\end{gather}
Here $\otimes$ is an outer product and
\eqref{1stderivative} means, for example, $\del_{r_{i}}\maru{r}(s) \in \mathrm{o}(s^{0})|s-t_{1}| t_{1}^{-1-\tilde{\epsilon}}$, and $\del_{\theta_{f}}\maru{\theta}(s) \in \mathrm{O}(1) t_{1}^{-\tilde{\epsilon}}$. 
\end{Thm}

A straightforward computation shows that $(\maru{r}, \maru{\theta}, \maru{\rho}, \maru{\omega})(s)$
satisfies the following integral equation:

\begin{gather}
(\maru{r}, \maru{\theta}, \maru{\rho}, \maru{\omega})(s)
=(P_{r}, P_{\theta}, P_{\rho}, P_{\omega})(\maru{r}, \maru{\theta}, \maru{\rho}, \maru{\omega})(s)	\label{integral_eq}:=\\
\left(
\begin{array}{c}
\int_{t_{1}}^{s} \bigl( \maru{\rho}(u) + F_{r}(u, r_{i}+(u-t_{1})\rho_{f}+\maru{r}(u), \theta_{f}+\maru{\theta}(u), \rho_{f}+\maru{\rho}(u), \omega_{i}+\maru{\omega}(u) ) \bigr)du\\
-\int_{s}^{t_{2}} \bigl(  F_{\theta}(u, r_{i}+(u-t_{1})\rho_{f}+\maru{r}(u), \theta_{f}+\maru{\theta}(u), \rho_{f}+\maru{\rho}(u), \omega_{i}+\maru{\omega}(u) ) \bigr)du\\
-\int_{s}^{t_{2}} \bigl(  F_{\rho}(u, r_{i}+(u-t_{1})\rho_{f}+\maru{r}(u), \theta_{f}+\maru{\theta}(u), \rho_{f}+\maru{\rho}(u), \omega_{i}+\maru{\omega}(u) ) \bigr)du\\
\int_{t_{1}}^{s} \bigl(  F_{\omega}(u, r_{i}+(u-t_{1})\rho_{f}+\maru{r}(u), \theta_{f}+\maru{\theta}(u), \rho_{f}+\maru{\rho}(u), \omega_{i}+\maru{\omega}(u) ) \bigr)du
\end{array}
\right) \notag
\end{gather}
where the map $P=(P_{r}, P_{\theta}, P_{\rho}, P_{\omega})$ depends on the parameters $t_{1}, t_{2},  r_{i}, \theta_{f}, \rho_{f}, \omega_{i}$.
We will apply the fixed point theorem to solve  \eqref{integral_eq}. 
We define the Banach space on which the map $P$ is defined as follows:

\begin{Def}\label{Banach_Space_of_functions_Newton}
For $ m \geq 0$, we define
\begin{gather*}
Z_{T}^{m}:=\{ z \in C([T, \infty )): \sup \frac{|z(t)|}{|t-T|^{m}} < \infty \}, \ \ 
Z_{T, \infty}^{m}:=\{ z \in Z_{T}^{m}: \lim_{t \to \infty} \frac{|z(t)|}{|t-T|^{m}} = 0 \} .
\end{gather*}
For $m < 0$, we define
\begin{gather*}
Z_{T}^{m}:=\{ z \in C([T, \infty )): \sup \frac{|z(t)|}{\langle t \rangle^{m}} < \infty \}.
\end{gather*}
We define
\begin{gather*}
Z_{t_{1}}^{1, 0, -\tilde{\epsilon}, 1- \tilde{\epsilon}} := \{ (\maru{r}, \maru{\theta}, \maru{\rho}, \maru{\omega}) \in Z_{t_{1}, \infty}^{1} \times Z_{t_{1}}^{0} \times Z_{t_{1}}^{-\tilde{\epsilon}} \times Z_{t_{1}, \infty}^{1- \tilde{\epsilon}} \}.
\end{gather*}
\end{Def}
Then we have the following Proposition:

\begin{Prop}\label{P:contraction}
For large enough $T > 0$,
the map $P$ is a contraction map on\\
 $Z_{t_{1}}^{1, 0, -\tilde{\epsilon}, 1- \tilde{\epsilon}}$
for any $T \leq t_{1} \leq t_{2} \leq \infty, (r_{i}, \theta_{f}, \rho_{f}, \omega_{i}) \in T^{*}(\mathbb{R} \times \mathbb{R}^{n-1})$.
Indeed, for some constant $c$ which does not depend on $t_{1}, t_{2}, (r_{i}, \theta_{f}, \rho_{f}, \omega_{i})$ but $T$, we have
\begin{gather}
\| \nabla_{x}P(x) \|_{B\left( Z_{t_{1}}^{1, 0, -\tilde{\epsilon}, 1- \tilde{\epsilon}} \right)} < c <1
\label{contraction}.
\end{gather}
\end{Prop}

\begin{proof}
We first note that $P$ is well defined as a map of $Z_{t_{1}}^{1, 0, -\tilde{\epsilon}, 1- \tilde{\epsilon}}$ into itself. 
Indeed, for example, if $x =(\maru{r}, \maru{\theta}, \maru{\rho}, \maru{\omega}) \in Z_{t_{1}}^{1, 0, -\tilde{\epsilon}, 1- \tilde{\epsilon}}$,
\begin{gather*}
|P_{r}(x)(s) |\\
\leq
\int_{t_{1}}^{s} |\bigl( \maru{\rho}(u) + F_{r}(u, r_{i}+(s-t_{1})\rho_{f}+\maru{r}(u), \theta_{f}+\maru{\theta}(u), \rho_{f}+\maru{\rho}(u), \omega_{i}+\maru{\omega}(u) ) \bigr)|du\\
\leq
\int_{t_{1}}^{s} |C \langle u \rangle^{-\tilde{\epsilon}}+C \langle u \rangle^{-\epsilon}|du,
\end{gather*}
which implies $P_{r}(x)(s) \in Z_{t_{1}, \infty}^{1}$.
Others are similar to prove.

Now we check that $P$ is a contraction on $Z_{t_{1}}^{1, 0, -\tilde{\epsilon}, 1- \tilde{\epsilon}}$.
It suffices to show 
 \eqref{contraction}
for some constant $c$ which does not depend on $t_{1}, t_{2}, (r_{i}, \theta_{f}, \rho_{f}, \omega_{i})$ but $T$.
Let $v \in Z_{t_{1}, \infty}^{1}$.
Then
\begin{align*}
 |s-t_{1}|^{-1} (\nabla_{r}P_{r}(x) v)(s) 
 &\leq |s-t_{1}|^{-1} \int_{t_{1}}^{s} \|\nabla_{r} F_{r}(u, \cdot )\|_{\infty} |u-t_{1}| \| v \|_{Z_{t_{1}, \infty}^{1}} du\\
&\leq \| v \|_{Z_{t_{1}, \infty}^{1}} \int_{t_{1}}^{s} |s - t_{1}|^{-1} |u - t_{1}| \langle u \rangle^{-1-\epsilon} du.
\end{align*}
If we let $T \to \infty$, the right hand side goes to zero uniformly for $T \leq t_{1} \leq t_{2} \leq \infty$. 
Moreover, the right hand side goes to zero as $s \to \infty$.
Hence taking $T$ large enough, we may assure that 
\begin{gather*}
\| \nabla_{r}P_{r} \|_{B(Z_{t_{1}, \infty}^{1})} < c < 1
\end{gather*}
for some constant $c$  for any $T \leq t_{1} \leq t_{2} \leq \infty$ and for any $(r_{i}, \theta_{f}, \rho_{f}, \omega_{i}) \in T^{*}(\mathbb{R} \times \mathbb{R}^{n-1})$. 
In a similar way, we can show that for some large enough $T$,  \eqref{contraction} holds for any $T \leq t_{1} \leq t_{2} \leq \infty$ and for any $(r_{i}, \theta_{f}, \rho_{f}, \omega_{i}) \in T^{*}(\mathbb{R} \times \mathbb{R}^{n-1})$. 
\end{proof}

\begin{proof}[\bf{Proof of Theorem \ref{E4Neq}} .]
The fixed point theorem together with Proposition  \ref{P:contraction} implies that there exists a unique solution
$(\maru{r}, \maru{\theta}, \maru{\rho}, \maru{\omega})(s) \in Z_{t_{1}}^{1, 0, -\tilde{\epsilon}, 1- \tilde{\epsilon}}$
for the integral equation  \eqref{integral_eq}
for each $T \leq t_{1} < t_{2} \leq \infty$
and 
$(r_{i}, \theta_{f}, \rho_{f}, \omega_{i}) \in T^{*}(\mathbb{R} \times \mathbb{R}^{n-1})$ if $T$ is large enough.
$(\maru{r}, \maru{\theta}, \maru{\rho}, \maru{\omega})(s) \in Z_{t_{1}}^{1, 0, -\tilde{\epsilon}, 1- \tilde{\epsilon}}$ directly means  \eqref{0thderivative}.

Let us now prove  \eqref{1stderivative}. We use the identity
\begin{align}
(I - \nabla_{x}P(x) ) \del^{\gamma}(x) 
&=h^{\gamma}=(h_{r}^{\gamma}, h_{\theta}^{\gamma}, h_{\rho}^{\gamma}, h_{\omega}^{\gamma})\label{1stderivatives}\\
:&=
\left(
\begin{array}{c}
\int_{t_{1}}^{s} (\nabla F_{r})(u, y)\del^{\gamma}(y-x)du\\
-\int_{s}^{t_{1}} (\nabla F_{\theta})(u, y)\del^{\gamma}(y-x)du\\
-\int_{s}^{t_{1}} (\nabla F_{\rho})(u, y)\del^{\gamma}(y-x)du\\
\int_{t_{1}}^{s} (\nabla F_{\omega})(u, y)\del^{\gamma}(y-x)du
\end{array}
\right) \notag
\end{align}
where $\del^{\gamma} = \del_{r_{i}}, \del_{\theta_{f}}, \del_{\rho_{f}},$ or $ \del_{\omega_{i}}$,
$x = (\maru{r}, \maru{\theta}, \maru{\rho}, \maru{\omega})$
is the solution of  \eqref{integral_eq}, and
$y = (r_{i}+(u-t_{1})\rho_{f}+\maru{r}(u), \theta_{f}+\maru{\theta}(u), \rho_{f}+\maru{\rho}(u), \omega_{i}+\maru{\omega}(u) )$.
By a straight computation we have
\begin{gather*}
(h^{\del_{r_{i}}}, h^{\del_{\theta_{i}}}, h^{\del_{\rho_{i}}}, h^{\del_{\omega_{i}}})
\in (\langle t_{1} \rangle^{-1-\tilde{\epsilon}}, \langle t_{1} \rangle^{-\tilde{\epsilon}}, \langle t_{1} \rangle^{-\tilde{\epsilon}}, \langle t_{1} \rangle^{-1} ) Z_{t_{1}}^{1, 0, -\tilde{\epsilon}, 1-\tilde{\epsilon}}.
\end{gather*}
 \eqref{contraction} implies that $I - \nabla_{x}P(x)$ is invertible on $Z_{t_{1}}^{1, 0, -\tilde{\epsilon}, 1-\tilde{\epsilon}}$.
Using  \eqref{1stderivatives}, we get 
\begin{gather*}
(\del_{r_{i}} x, \del_{\theta_{f}}x, \del_{\rho_{f}}x, \del_{\omega_{i}}x) \in Z_{t_{1}}^{1, 0, -\tilde{\epsilon}, 1-\tilde{\epsilon}},
\end{gather*}
and 
\begin{gather*}
\| (\del_{r_{i}} x, \del_{\theta_{f}}x, \del_{\rho_{f}}x, \del_{\omega_{i}}x)\|_{Z_{t_{1}}^{1, 0, -\tilde{\epsilon}, 1-\tilde{\epsilon}}} \in \mathrm{O}(\langle t_{1} \rangle^{-1-\tilde{\epsilon}}, \langle t_{1} \rangle^{-\tilde{\epsilon}}, \langle t_{1} \rangle^{-\tilde{\epsilon}}, \langle t_{1} \rangle^{-1} ),
\end{gather*}
which implies  \eqref{1stderivative}.

Now we prove  \eqref{nthderivative} by an induction. Assume  that 
$\del^{\gamma}=\del_{r_{i}}^{l} \del_{\theta}^{\alpha} \del_{\rho}^{k} \del_{\omega}^{\beta}$,
$l+|\alpha |+ k+|\beta | = n \geq 2,$
and  \eqref{nthderivative} is true for $l+|\alpha|+ k+|\beta| \leq n-1$.
We use the identity
\begin{align}
(I - \nabla_{x}P(x) ) \del^{\gamma}(x) 
&=h^{\gamma}=(h_{r}^{\gamma}, h_{\theta}^{\gamma}, h_{\rho}^{\gamma}, h_{\omega}^{\gamma})\label{nthderivatives}\\
:&=
\left(
\begin{array}{c}
\int_{t_{1}}^{s} \sum_{q \geq 2}(\nabla^{q} F_{r})(u, y)\del^{\gamma_{1}}(y)\del^{\gamma_{2}}(y)\cdots \del^{\gamma_{q}}(y)du\\
-\int_{s}^{t_{1}} \sum_{q \geq 2}(\nabla^{q} F_{\theta})(u, y)\del^{\gamma_{1}}(y)\del^{\gamma_{2}}(y)\cdots \del^{\gamma_{q}}(y)du\\
-\int_{s}^{t_{1}} \sum_{q \geq 2}(\nabla^{q} F_{\rho})(u, y)\del^{\gamma_{1}}(y)\del^{\gamma_{2}}(y)\cdots \del^{\gamma_{q}}(y)du\\
\int_{t_{1}}^{s} \sum_{q \geq 2}(\nabla^{q} F_{\omega})(u, y)\del^{\gamma_{1}}(y)\del^{\gamma_{2}}(y)\cdots \del^{\gamma_{q}}(y)du
\end{array}
\right) \notag
\end{align}
where the sum is taken over $\gamma = \sum_{p=1}^{q} \gamma_{p}, q \geq 2$.
The induction hypothesis with a straight computation shows that 
\begin{gather*}
(h^{\gamma})
\in (\langle t_{1} \rangle^{-l - |\beta |} ) Z_{t_{1}}^{1, 0, -\tilde{\epsilon}, 1-\tilde{\epsilon}}.
\end{gather*}
Thus we have 
\begin{gather*}
\| \del_{\gamma} x \|_{Z_{t_{1}}^{1, 0, -\tilde{\epsilon}, 1-\tilde{\epsilon}}} \in (\langle t_{1} \rangle^{-l - |\beta|} ),
\end{gather*}
which implies  \eqref{nthderivative}.
\end{proof}

\subsection{Classical trajectories with long-range time-independent force}\label{Classical trajectories with long-range time-independent force}

We denote the outgoing region by $\Gamma_{R, U, J, Q}^{\  +. \tilde{\epsilon}}$:
\begin{gather*}
\Gamma_{R, U, J, Q}^{\  +. \tilde{\epsilon}}
:=
\{ (r, \theta, \rho, \omega) \in T^{*}(\mathbb{R}\times \mathbb{R}^{n-1}):
r > R,
\theta \in U,
\rho \in J,
|\omega| \leq Q r^{1-\tilde{\epsilon}}
\}
\end{gather*}
for $R>0, U \subset \mathbb{R}^{n-1}, J\subset\mathbb{R}, Q>0$.

We now consider the dynamics with time-independent long-range forces.
\begin{Def}\label{long-range_force}
A time-independent force $F$ is said to be a long-range force if it satisfies
\begin{gather}
\sup _{(r, \theta, \rho, \omega) \in \Gamma_{R, U, J, Q}^{\  +. \tilde{\epsilon}} }
|\del_{r}^{l} \del_{\theta}^{\alpha} \del_{\rho}^{k} \del_{\omega}^{\beta}(F_{r}, F_{\theta}, F_{\rho}, F_{\omega})(r, \theta, \rho, \omega ) |
\in \mathrm{O}( \langle R \rangle^{-n_{r, \theta ,\rho ,\omega }(l, \alpha, k, \beta) })
\label{LRFAss}
\end{gather}
for any $R > 0$, $U \Subset \mathbb{R}^{n-1}$, $J \Subset (0, \infty)$, $Q > 0$.
\end{Def}
As in Theorem \ref{E4Neq}, we show the unique existence of trajectories for the dynamics  where the boundary conditions are the initial position and the final momentum.

\begin{Thm}\label{CTini-fi}
Assume that $F$ is a time-independent long-range force in the sense of Definition \ref{long-range_force}.
Then for any open
$U \Subset \tilde{U} \Subset \mathbb{R}^{n-1}$, open
$J \Subset \tilde{J} \Subset (0, \infty)$, and
$Q > 0$,
there exists $R>0$ such that 
for any $t \geq 0$ and for any
$(r_{i}, \theta_{f}, \rho_{f}, \omega_{i}) \in \Gamma_{R, U, J, Q}^{\  +. \tilde{\epsilon}}$,
there exists a unique trajectory
\begin{gather*}
[0, t] \ni s \mapsto (\tilde{r}, \tilde{\theta}, \tilde{\rho}, \tilde{\omega})(s, t,  r_{i}, \theta_{f}, \rho_{f}, \omega_{i})
\end{gather*}
satisfying
\begin{gather}
\del_{s}(\tilde{r}, \tilde{\theta}, \tilde{\rho}, \tilde{\omega})(s, t,  r_{i}, \theta_{f}, \rho_{f}, \omega_{i})
= (\rho + F_{r}, F_{\theta}, F_{\rho}, F_{\omega})((\tilde{r}, \tilde{\theta}, \tilde{\rho}, \tilde{\omega})(s, t,  r_{i}, \theta_{f}, \rho_{f}, \omega_{i}) ) \label{time_dependent_Newton_eq1}\\
(\tilde{r}, \tilde{\omega})(0, t, r_{i}, \theta_{f}, \rho_{f}, \omega_{i})
= (r_{i}, \omega_{i}), \ \ 
(\tilde{\theta}, \tilde{\rho})(t, t, r_{i}, \theta_{f}, \rho_{f}, \omega_{i})
= (\theta_{f}, \rho_{f}),\notag
\end{gather}
and the estimates
\begin{gather}
|\maru{r}(s)| \in \mathrm{o}((s+ \langle r_{i} \rangle)^{0}) |s|, \ \ 
|\maru{\theta}(s)| \in \mathrm{O}((s+ \langle r_{i} \rangle)^{-\tilde{\epsilon}}),	\label{estimates_for_uniqueness}\\
|\maru{\rho}(s)| \in \mathrm{O}((s+ \langle r_{i} \rangle)^{-\tilde{\epsilon}}), \ \ 
|\maru{\omega}(s)| \in \mathrm{o}((s+ \langle r_{i} \rangle)^{0}) |s|^{1-\tilde{\epsilon}}, \notag
\end{gather}
and
\begin{gather*}
\tilde{\theta}(s, t,  r_{i}, \theta_{f}, \rho_{f}, \omega_{i}) \in \tilde{U}, \ \ 
\tilde{\rho}(s, t,  r_{i}, \theta_{f}, \rho_{f}, \omega_{i})\in \tilde{J}
\end{gather*}
where
\begin{gather}
\maru{r}(s)		=\tilde{r}(s) - r_{i} - s \rho_{f}, \ \ 
\maru{\theta}(s)=\tilde{\theta}(s) - \theta_{f}\notag\\
\maru{\rho}(s)	=\tilde{\rho}(s) - \rho_{f}, \ \ 
\maru{\omega}(s)=\tilde{\omega}(s) - \omega_{i}.\notag
\end{gather}
Moreover the solution satisfies the following estimates uniformly\\ for 
$0 \leq s \leq t \leq \infty$, 
$(r_{i}, \theta_{f}, \rho_{f}, \omega_{i}) \in \Gamma_{R, U, J, Q}^{\  +. \tilde{\epsilon}}$ :
\begin{gather*}
\left(
\begin{array}{cccc}
\del_{r_{i}}\maru{r}(s)			&
\del_{\theta_{f}}\maru{r}(s)	&
\del_{\rho_{f}}\maru{r}(s)		&
\del_{\omega_{i}}\maru{r}(s)	\\
\del_{r_{i}}\maru{\theta}(s)			&
\del_{\theta_{f}}\maru{\theta}(s)	&
\del_{\rho_{f}}\maru{\theta}(s)		&
\del_{\omega_{i}}\maru{\theta}(s)	\\
\del_{r_{i}}\maru{\rho}(s)			&
\del_{\theta_{f}}\maru{\rho}(s)	&
\del_{\rho_{f}}\maru{\rho}(s)		&
\del_{\omega_{i}}\maru{\rho}(s)	\\
\del_{r_{i}}\maru{\omega}(s)			&
\del_{\theta_{f}}\maru{\omega}(s)	&
\del_{\rho_{f}}\maru{\omega}(s)		&
\del_{\omega_{i}}\maru{\omega}(s)	
\end{array}
\right)			\\
\in
\left(
\begin{array}{c}
\mathrm{o}((s+ \langle r_{i} \rangle)^{0}) |s|		\\
\mathrm{O}(1)								\\
\mathrm{O}((s+ \langle r_{i} \rangle)^{-\tilde{\epsilon}})			\\
\mathrm{o}((s+ \langle r_{i} \rangle)^{0}) |s|^{1-\tilde{\epsilon}}
\end{array}
\right)
\otimes
(\langle r_{1} \rangle^{-1-\tilde{\epsilon}},
\langle r_{1} \rangle^{- \tilde{\epsilon}},
\langle r_{1} \rangle^{- \tilde{\epsilon}},
\langle r_{1} \rangle^{-1}
),
\end{gather*}
\begin{gather*}
\del_{r_{i}}^{l} \del_{\theta}^{\alpha} \del_{\rho}^{k} \del_{\omega}^{\beta}
\left(
\begin{array}{c}
\maru{r}\\
\maru{\theta}\\
\maru{\rho}\\
\maru{\omega}
\end{array}
\right)
\in
\left(
\begin{array}{c}
\mathrm{o}((s+ \langle r_{i} \rangle)^{0}) |s-t_{1}|		\\
\mathrm{O}(1)								\\
\mathrm{O}((s+ \langle r_{i} \rangle)^{-\tilde{\epsilon}})			\\
\mathrm{o}((s+ \langle r_{i} \rangle)^{0}) |s-t_{1}|^{1-\tilde{\epsilon}}
\end{array}
\right) \cdot
\langle r_{i} \rangle^{-l-|\beta |}.
\end{gather*}

\end{Thm}

\begin{proof}
There exists $C_{0}$ such that
if $\rho \in J$ and $r > 0$, then
\begin{gather}
|r +(s-t_{1})\rho| \geq C_{0}(|s-t_{1}|+r) \label{outgoing}.
\end{gather}
We fix constants $\epsilon_{0}, \tilde{Q}, \epsilon_{1}$ such that
\begin{gather*}
0 < \epsilon_{0} < C_{0},\ \ 
\tilde{Q} \geq \frac{2Q}{C_{0}^{1-\tilde{\epsilon}}},\ \ 
0 < \epsilon_{1} < \frac{1}{2}\tilde{Q}(C_{0}-\epsilon_{0})^{1-\tilde{\epsilon}},
\end{gather*}
and introduce cut-off functions $I_{r}, I_{\theta}, I_{\rho}, I_{\omega}$ as follows.
We take $I_{r} \in C^{\infty}(0, \infty)$ such that $I_{r} = 1$ on a neighborhood of $\{r; r > C_{0}-\epsilon_{0} \}$, $I_{\theta} \in C_{0}^{\infty}(\mathbb{R}^{n-1})$ such that $I_{\theta} = 1 $ on $\tilde{U}$, $I_{\rho} \in C_{0}^{\infty}(0, \infty)$ such that $I_{\rho} = 1$ on $\tilde{J}$, and $I_{\omega} \in C_{0}^{\infty}(\mathbb{R}^{n-1})$ such that
$I_{\omega} = 1$ on a neighborhood of $\{ \omega: |\omega| < \tilde{Q}\}$.
Using these cut-off functions, we define the effective time-dependent force $F_{\text{e}}$ by
\begin{gather*}
F_{\text{e}}(t, r, \theta, \rho, \omega)
= I_{r}(\frac{r}{t}) I_{\theta}(\theta ) I_{\rho}( \rho ) I_{\omega}(\frac{\omega}{r^{1-\tilde{\epsilon}}})F(r, \theta, \rho, \omega).
\end{gather*}
It follows from  \eqref{LRFAss} that $F_{\text{e}}(t, r, \theta, \rho, \omega)$ 
is a slowly-decaying force in the sense of Definition \ref{slowly-decaying_force}.
Therefore, we can find $T$ such that the boundary value problem considered in Theorem \ref{E4Neq}
possesses a unique solution for any $T \leq t_{1} \leq t_{2}$ 
and any $r_{i}, \theta_{f}, \rho_{f}, \omega_{i}$.
Let us denote this solution by
\begin{gather*}
(\tilde{r}_{e}, \tilde{\theta}_{e}, \tilde{\rho}_{e}, \tilde{\omega}_{e})(s, t_{1}, t_{2}, r_{i}, \theta_{f}, \rho_{f}, \omega_{i}).
\end{gather*}
By enlarging $T$ if needed, we can guarantee that
\begin{gather}
|\tilde{r}_{e}(s, t_{1}, t_{2}, r_{i}, \theta_{f}, \rho_{f}, \omega_{i}) - r_{i} - (s- t_{1}) \rho_{f}| \leq \epsilon_{0}|s-t_{1}|,\label{rissmall}\\
|\tilde{\theta}_{e}(s, t_{1}, t_{2}, r_{i}, \theta_{f}, \rho_{f}, \omega_{i}) - \theta_{f}| \leq
 \text{dist}(U, \tilde{U}^{C}) \notag \\
|\tilde{\rho}_{e}(s, t_{1}, t_{2}, r_{i}, \theta_{f}, \rho_{f}, \omega_{i}) - \rho_{f}| \leq
 \text{dist}(J, \tilde{J}^{C}) \notag \\ 
 |\tilde{\omega}_{e}(s, t_{1}, t_{2}, r_{i}, \theta_{f}, \rho_{f}, \omega_{i}) - \omega_{i}| \leq \epsilon_{1}|s-t_{1}|^{1- \tilde{\epsilon}}\label{omegaissmall}.
\end{gather}

We claim that if $R =T(C_{0}-\epsilon_{0})/C_{0}$
and $(r_{i}, \theta_{f}, \rho_{f}, \omega_{i}) \in\Gamma_{R, U, J, Q}^{\  +. \tilde{\epsilon}},$
then we can solve our boundary problem by setting
\begin{gather}
(\tilde{r}, \tilde{\theta}, \tilde{\rho}, \tilde{\omega})(s, t, r_{i}, \theta{f}, \rho_{f}, \omega_{i}) 
:= (\tilde{r}_{e}, \tilde{\theta}_{e}, \tilde{\rho}_{e}, \tilde{\omega}_{e})(r + s, r, r+t, r_{i}, \theta_{f}, \rho_{f}, \omega_{i}) \label{time_dependent_independent_trick_solution}
\end{gather}
where
$r=|r_{i}|C_{0}/(C_{0}-\epsilon).$
Indeed, from  \eqref{outgoing},  \eqref{rissmall}, and  \eqref{omegaissmall} we see that
\begin{gather*}
|\tilde{r}_{e}(r + s, r, r+t, r_{i}, \theta_{f}, \rho_{f}, \omega_{i})|
\geq (C_{0} - \epsilon_{0})|s + r|,\\
\tilde{\theta}_{e}(r + s, r, r+t, r_{i}, \theta_{f}, \rho_{f}, \omega_{i}) \in \tilde{U},\\
\tilde{\rho}_{e}(r + s, r, r+t, r_{i}, \theta_{f}, \rho_{f}, \omega_{i}) \in \tilde{J},
\end{gather*}
and
\begin{gather*}
|\tilde{\omega}_{e}(r + s, r, r+t, r_{i}, \theta_{f}, \rho_{f}, \omega_{i})|
\leq \epsilon_{1}|s|^{1- \tilde{\epsilon}} + \omega_{i}
\leq \epsilon_{1}|s|^{1- \tilde{\epsilon}} + Q r_{i}^{1-\tilde{\epsilon}}\\
\leq \epsilon_{1}|s|^{1- \tilde{\epsilon}} + Q \bigl( \frac{C_{0} - \epsilon_{0}}{C_{0}} \bigr)^{1-\tilde{\epsilon}} r_{i}^{1-\tilde{\epsilon}}
\leq \tilde{Q}(C_{0} - \epsilon_{0})^{1-\tilde{\epsilon}} |s + r |^{1-\tilde{\epsilon}}\\
\leq \tilde{Q} |\tilde{r}_{e}(r + s, r, r+t, r_{i}, \theta_{f}, \rho_{f}, \omega_{i})|^{1-\tilde{\epsilon}}.
\end{gather*}
Hence we have
\begin{gather*}
F_{\text{e}}(r + s, (\tilde{r}_{e}, \tilde{\theta}_{e}, \tilde{\rho}_{e}, \tilde{\omega}_{e})(r + s, r, r+t, r_{i}, \theta_{f}, \rho_{f}, \omega_{i}) ) \\
= F( (\tilde{r}_{e}, \tilde{\theta}_{e}, \tilde{\rho}_{e}, \tilde{\omega}_{e})(r + s, r, r+t, r_{i}, \theta_{f}, \rho_{f}, \omega_{i}) ).
\end{gather*}
Therefore the function  \eqref{time_dependent_independent_trick_solution}
solves the boundary problem  \eqref{time_dependent_Newton_eq1} with the initial time-independent force.

The estimates on $(\tilde{r}, \tilde{\theta}, \tilde{\rho}, \tilde{\omega})(s, t, r_{i}, \theta{f}, \rho_{f}, \omega_{i}) $ are obtained directly from \\those of Theorem  \ref{E4Neq} 
using the identity  \eqref{time_dependent_independent_trick_solution} and replacing $s, t_{1}, t_{2}$
there \\
by $s + \langle r_{i} \rangle, \langle r_{i} \rangle, t+\langle r_{i} \rangle$.

Finally, the uniqueness of the solution comes from the fact that any solution of  \eqref{time_dependent_Newton_eq1}
with  \eqref{estimates_for_uniqueness} is also a solution of the problem considered in Theorem \ref{E4Neq}
for the force $F_{\text{e}}(t, r, \theta, \rho, \omega)$ if time $t$ is large enough.

\end{proof}

Now we solve the dynamics with initial conditions.

\begin{Thm}\label{CTini}
Assume $F$ is a time-independent long-range force in the sense of Definition  \ref{long-range_force}.
Then for any open
$U \Subset \tilde{U} \Subset \mathbb{R}^{n-1}$, open
$J \Subset \tilde{J} \Subset (0, \infty)$, and
$Q > 0$,
there exists $R>0$ such that 
for any
$(r_{0}, \theta_{0}, \rho_{0}, \omega_{0}) \in \Gamma_{R, U, J, Q}^{\  +. \tilde{\epsilon}}$,
there exists a unique trajectory
\begin{gather*}
[0, \infty) \ni s \mapsto (\tilde{r}, \tilde{\theta}, \tilde{\rho}, \tilde{\omega})(s,  r_{0}, \theta_{0}, \rho_{0}, \omega_{0})
\end{gather*}
satisfying
\begin{gather*}
\del_{s}(\tilde{r}, \tilde{\theta}, \tilde{\rho}, \tilde{\omega})(s,  r_{0}, \theta_{0}, \rho_{0}, \omega_{0})
= (\rho + F_{r}, F_{\theta}, F_{\rho}, F_{\omega})((\tilde{r}, \tilde{\theta}, \tilde{\rho}, \tilde{\omega})(s,  r_{0}, \theta_{0}, \rho_{0}, \omega_{0}) ),\\
(\tilde{r}, \tilde{\theta}, \tilde{\rho}, \tilde{\omega})(0, r_{0}, \theta_{0}, \rho_{0}, \omega_{0})
= (r_{0}, \theta_{0}, \rho_{0}, \omega_{0}).
\end{gather*}
Set
\begin{gather*}
\maru{r}(s)		=\tilde{r}(s) - r_{0} - s \rho_{0} ,\ \ 
\maru{\theta}(s)=\tilde{\theta}(s) - \theta_{0},\\
\maru{\rho}(s)	=\tilde{\rho}(s) - \rho_{0},\ \ 
\maru{\omega}(s)=\tilde{\omega}(s) - \omega_{0}.
\end{gather*}
Moreover the solution satisfies the following estimates uniformly for 
$0 \leq s \leq t \leq \infty$, \\
$(r_{0}, \theta_{0}, \rho_{0}, \omega_{0}) \in \Gamma_{R, U, J, Q}^{\  +. \tilde{\epsilon}}$ :
\begin{gather*}
\tilde{\theta}(s,  r_{0}, \theta_{0}, \rho_{0}, \omega_{0}) \in \tilde{U},\ \ 
\tilde{\rho}(s,  r_{0}, \theta_{0}, \rho_{0}, \omega_{0}) \in \tilde{J},
\end{gather*}
\begin{gather*}
|\maru{r}(s)| \in \mathrm{o}((s+ \langle r_{0} \rangle)^{0}) |s|, \ \ 
|\maru{\theta}(s)| \in \mathrm{O}((\langle r_{0} \rangle)^{-\tilde{\epsilon}}),			\\
|\maru{\rho}(s)| \in \mathrm{O}((\langle r_{0} \rangle)^{-\tilde{\epsilon}}), \ \
|\maru{\omega}(s)| \in \mathrm{o}((s+ \langle r_{0} \rangle)^{0}) |s|^{1-\tilde{\epsilon}}, \\
\left(
\begin{array}{cccc}
\del_{r_{0}}\maru{r}(s)			&
\del_{\theta_{0}}\maru{r}(s)	&
\del_{\rho_{0}}\maru{r}(s)		&
\del_{\omega_{0}}\maru{r}(s)	\\
\del_{r_{0}}\maru{\theta}(s)			&
\del_{\theta_{0}}\maru{\theta}(s)	&
\del_{\rho_{0}}\maru{\theta}(s)		&
\del_{\omega_{0}}\maru{\theta}(s)	\\
\del_{r_{0}}\maru{\rho}(s)			&
\del_{\theta_{0}}\maru{\rho}(s)	&
\del_{\rho_{0}}\maru{\rho}(s)		&
\del_{\omega_{0}}\maru{\rho}(s)	\\
\del_{r_{0}}\maru{\omega}(s)			&
\del_{\theta_{0}}\maru{\omega}(s)	&
\del_{\rho_{0}}\maru{\omega}(s)		&
\del_{\omega_{0}}\maru{\omega}(s)	
\end{array}
\right)			\\
\in
\left(
\begin{array}{c}
\mathrm{o}((s+ \langle r_{0} \rangle)^{0}) |s|		\\
\mathrm{O}(1)								\\
(\langle r_{0} \rangle)^{-\tilde{\epsilon}}			\\
\mathrm{o}((s+ \langle r_{0} \rangle)^{0}) |s|^{1-\tilde{\epsilon}}
\end{array}
\right) 
\otimes
(\langle r_{0} \rangle^{-1-\tilde{\epsilon} },
\langle r_{0} \rangle^{\tilde{\epsilon} },
\langle r_{0} \rangle^{\tilde{\epsilon} },
\langle r_{0} \rangle^{-1}
),
\end{gather*}
\begin{gather*}
\del_{r_{0}}^{l} \del_{\theta_{0}}^{\alpha} \del_{\rho_{0}}^{k} \del_{\omega_{0}}^{\beta}
\left(
\begin{array}{c}
\maru{r}\\
\maru{\theta}\\
\maru{\rho}\\
\maru{\omega}
\end{array}
\right)
\in
\left(
\begin{array}{c}
\mathrm{o}((s+ \langle r_{0} \rangle)^{0}) |s-t_{1}|		\\
\mathrm{O}(1)								\\
(\langle r_{0} \rangle)^{-\tilde{\epsilon}}			\\
\mathrm{o}((s+ \langle r_{0} \rangle)^{0}) |s-t_{1}|^{1-\tilde{\epsilon}}
\end{array}
\right) \cdot
\langle r_{0} \rangle^{-l-|\beta |}.
\end{gather*}

\end{Thm}

\begin{proof}

Let $(\bar{r}, \bar{\theta}, \bar{\rho}, \bar{\omega})$
be the solutions in Theorem  \ref{CTini-fi} with $t=\infty$:
\begin{gather*}
[0, \infty ]\ni s \mapsto (\bar{r}, \bar{\theta}, \bar{\rho}, \bar{\omega})(s, \infty, r_{i}, \theta_{f}, \rho_{f}, \omega_{i}),\\
\del_{s}(\bar{r}, \bar{\theta}, \bar{\rho}, \bar{\omega})(s, \infty,  r_{i}, \theta_{f}, \rho_{f}, \omega_{i})\\
= (\rho + F_{r}, F_{\theta}, F_{\rho}, F_{\omega})((\bar{r}, \bar{\theta}, \bar{\rho}, \bar{\omega})(s, \infty,  r_{i}, \theta_{f}, \rho_{f}, \omega_{i}) ), \\
(\bar{r}, \bar{\omega})(0, \infty, r_{i}, \theta_{f}, \rho_{f}, \omega_{i})
= (r_{i}, \omega_{i}), \ \ 
(\bar{\theta}, \bar{\rho})(\infty, \infty, r_{i}, \theta_{f}, \rho_{f}, \omega_{i})
= (\theta_{f}, \rho_{f}).
\end{gather*}
Set
\begin{gather*}
(r_{0}, \theta_{0}, \rho_{0}, \omega_{0})(r_{i}, \theta_{f}, \rho_{f}, \omega_{i}) :=
(\bar{r}, \bar{\theta}, \bar{\rho}, \bar{\omega})(0, \infty, r_{i}, \theta_{f}, \rho_{f}, \omega_{i}).
\end{gather*}
It is clear that
\begin{gather*}
(r_{0}, \omega_{0})(r_{i}, \theta_{f}, \rho_{f}, \omega_{i})= (r_{i}, \omega_{i}).
\end{gather*}
Theorem  \ref{CTini-fi} assures the following estimates:
\begin{gather*}
|\theta_{0}(r_{i}, \theta_{f}, \rho_{f}, \omega_{i}) - \theta_{f}| \in \mathrm{O}( \langle r_{i} \rangle^{- \tilde{\epsilon}}), \ \ 
|\rho_{0}(r_{i}, \theta_{f}, \rho_{f}, \omega_{i}) - \rho_{f}| \in \mathrm{O}( \langle r_{i} \rangle^{- \tilde{\epsilon}}),
\end{gather*}
\begin{gather}
\left(
\begin{array}{cccc}
\del_{r_{i}}(\theta_{0}-\theta_{f})			&
\del_{\theta_{f}}(\theta_{0}-\theta_{f})	&
\del_{\rho_{f}}(\theta_{0}-\theta_{f})		&
\del_{\omega_{i}}(\theta_{0}-\theta_{f})	\\
\del_{r_{i}}(\rho_{0}-\rho_{f})			&
\del_{\theta_{f}}(\rho_{0}-\rho_{f})	&
\del_{\rho_{f}}(\rho_{0}-\rho_{f})		&
\del_{\omega_{i}}(\rho_{0}-\rho_{f})	
\end{array}
\right)			\notag\\  
\in
\left(
\begin{array}{c}
\mathrm{O}(1)								\\
(\langle r_{i} \rangle)^{-\tilde{\epsilon}}			
\end{array}
\right)
\otimes
(\langle r_{1} \rangle^{-1-\tilde{\epsilon}},
\langle r_{1} \rangle^{- \tilde{\epsilon}},
\langle r_{1} \rangle^{- \tilde{\epsilon}},
\langle r_{1} \rangle^{-1}
).
\label{estimates_for_ini-fi}
\end{gather}
By taking $R$ large enough, we can assure that the map $(r_{0}, \theta_{0}, \rho_{0}, \omega_{0})(r_{i}, \theta_{f}, \rho_{f}, \omega_{i})$ is injective.
Let $(r_{i}, \theta_{i}, \rho_{i}, \omega_{i})(r_{0}, \theta_{0}, \rho_{0}, \omega_{0})$
be the inverse function. 
We will show that
\begin{gather*}
(\tilde{r}, \tilde{\theta}, \tilde{\rho}, \tilde{\omega})(s, r_{0}, \theta_{0}, \rho_{0}, \omega_{0}):=
(\bar{r}, \bar{\theta}, \bar{\rho}, \bar{\omega})(s, \infty, (r_{i}, \theta_{f}, \rho_{f}, \omega_{i})(r_{0}, \theta_{0}, \rho_{0}, \omega_{0}))
\end{gather*}
gives the desired function. \eqref{estimates_for_ini-fi} implies
\begin{gather}
\left(
\begin{array}{cccc}
\del_{r_{0}}(\theta_{f}-\theta_{0})			&
\del_{\theta_{0}}(\theta_{f}-\theta_{0})	&
\del_{\rho_{0}}(\theta_{f}-\theta_{0})		&
\del_{\omega_{0}}(\theta_{f}-\theta_{0})	\\
\del_{r_{0}}(\rho_{f}-\rho_{0})			&
\del_{\theta_{0}}(\rho_{f}-\rho_{0})	&
\del_{\rho_{0}}(\rho_{f}-\rho_{0})		&
\del_{\omega_{0}}(\rho_{f}-\rho_{0})	
\end{array}
\right)			\notag\\  
\in
\left(
\begin{array}{c}
\mathrm{O}(1)								\\
(\langle r_{0} \rangle)^{-\tilde{\epsilon}}			
\end{array}
\right) 
\otimes
(\langle r_{1} \rangle^{-1-\tilde{\epsilon}},
\langle r_{1} \rangle^{- \tilde{\epsilon}},
\langle r_{1} \rangle^{- \tilde{\epsilon}},
\langle r_{1} \rangle^{-1}
).
\label{estimates_for_initial_1}
\end{gather}
Moreover, it is easy to see that
\begin{gather}
\del_{r_{0}}^{l} \del_{\theta_{0}}^{\alpha} \del_{\rho_{0}}^{k} \del_{\omega_{0}}^{\beta}
\left(
\begin{array}{c}
\theta_{f}-\theta_{0} \\
\rho_{f}-\rho_{0}
\end{array}
\right)
\in 
\left(
\begin{array}{c}
\mathrm{O}(\langle r_{0} \rangle^{-l-|\beta |})\\
\mathrm{O}(\langle r_{0} \rangle^{-l-|\beta | -\tilde{\epsilon}})
\end{array}
\right). \label{estimates_for_initial_2}
\end{gather}
 \eqref{estimates_for_initial_1} and  \eqref{estimates_for_initial_2}
shows the desired estimates.
\end{proof}

\subsection{Solutions to the Hamilton-Jacobi equation}\label{Solutions of the Hamilton-Jacobi equation}

We state a lemma which relates the hamiltonian $h$ with the time-independent force $F$.

\begin{Lem}\label{HtoF}
Let 
\begin{gather*}
h(r, \theta, \rho, \omega ) = \frac{1}{2}\rho^{2} + \tilde{h}(r, \theta, \rho, \omega )\\
\tilde{h}(r, \theta, \rho, \omega )
= \frac{1}{2} a_{1}(r, \theta ) \rho^{2} + \frac{1}{r} a_{2}^{j}\rho \omega_{j}
+ \frac{1}{2r^{2}} a_{3}^{j k} \omega_{j} \omega_{k} +V^{L}(r, \theta) .
\end{gather*}
Assume
\begin{gather*}
|\partial_{r}^{l}\partial_{\theta}^{\alpha}a_{j}(r, \theta ) | \leq C_{l, \alpha} r^{-\mu_{j}-l}, \ \ 
|D_{r}^{j}D_{\theta}^{\alpha}V^{L}(r, \theta )|\leq C_{j} r^{-\mu_{L} -j},
\end{gather*}
with
\begin{gather*}
\mu_{1} = \mu_{2} = \mu_{L} = \epsilon > 0, \ \ 
\mu_{3} = 0.
\end{gather*}
Then for any
$ U \Subset \mathbb{R}^{n-1}, J\Subset\mathbb{R},$ and $ Q>0$,
\begin{gather}
\sup _{(r, \theta, \rho, \omega) \in \Gamma_{R, U, J, Q}^{\  +. \tilde{\epsilon}} }
|\del_{r}^{l} \del_{\theta}^{\alpha} \del_{\rho}^{k} \del_{\omega}^{\beta}(\tilde{h})(r, \theta, \rho, \omega ) |
\in \mathrm{O}( \langle R \rangle^{-m(l, \alpha, k, \beta) })
\label{LRhAss}.
\end{gather}
This immediately implies that setting
\begin{gather}
(F_{r}, F_{\theta}, F_{\rho}, F_{\omega}) = (\del_{\rho}\tilde{h}, \del_{\omega}\tilde{h}, -\del_{r}\tilde{h}, -\del_{\theta}\tilde{h}),\label{Ftohformula}
\end{gather}
we have
 \eqref{LRFAss}, i.e., $h$ defines a long-range time-independent force via \eqref{Ftohformula}.\end{Lem}

Combining Theorem  \ref{CTini} and Lemma  \ref{HtoF}, we obtain solutions to the Hamilton-Jacobi equation:

\begin{Thm}\label{SThm}
Let $h, \tilde{h}$ be as in Lemma  \ref{HtoF}.
For any $\tilde{U} \Subset U \Subset \mathbb{R}^{n-1}, \tilde{J} \Subset J \Subset (0, \infty)$, $C_{j, \alpha} > 0$, there exists $T > 0$ such that if a smooth function $\psi(\rho, \theta)$ defined on $J \times U$ satisfies
\begin{gather}
|\del_{\rho}^{j}\del_{\theta}^{\alpha} ( \psi(\rho, \theta) - \frac{1}{2}s \rho^{2} )| \leq C_{j, \alpha} \langle s \rangle^{1 - \tilde{\epsilon}} \label{psi_condition}
\end{gather}
for some $s > T$, then there exists a unique function $S(t, \rho, \theta)$
defined on a region 
$\Theta \subset(0, \infty ) \times (0, \infty) \times \mathbb{R}^{n-1}$ ( which will be defined in the proof ), with
$\Theta \supset(0, \infty ) \times \tilde{J} \times \tilde{U}$,
satifying the Hamilton-Jacobi equation:
\begin{gather*}
(\del_{t} S )(t, \rho, \theta ) = h( (\del_{\rho}S )(\rho, \theta ) , \theta, \rho, -(\del_{\theta}S )(\rho, \theta ) )
\end{gather*}
with the initial value
\begin{gather*}
S(0, \rho, \theta) = \psi(\rho, \theta).
\end{gather*}
Moreover the function $S$ satisfies the following estimates:
\begin{gather}
|\del_{\rho}^{j} \del_{\theta}^{\alpha} (S(t, \rho, \theta ) - \frac{1}{2}t \rho^{2} )|
\leq \tilde{C}_{j, \alpha} \langle t \rangle^{1- \epsilon}. \label{estimates_for_S}
\end{gather}
\end{Thm}

\begin{proof}

Let
\begin{gather*}
[0, \infty) \ni t \mapsto (\tilde{r}, \tilde{\theta}, \tilde{\rho}, \tilde{\omega})(t,  r_{0}, \theta_{0}, \rho_{0}, \omega_{0})
\end{gather*}
be the unique trajectory of the Hamilton equations with initial value problem as in the Theorem  \ref{CTini}:
\begin{gather*}
\del_{t}(\tilde{r}, \tilde{\theta}, \tilde{\rho}, \tilde{\omega})(t,  r_{0}, \theta_{0}, \rho_{0}, \omega_{0})
= (\rho + F_{r}, F_{\theta}, F_{\rho}, F_{\omega})((\tilde{r}, \tilde{\theta}, \tilde{\rho}, \tilde{\omega})(t,  r_{0}, \theta_{0}, \rho_{0}, \omega_{0}) ),\\
(\tilde{r}, \tilde{\theta}, \tilde{\rho}, \tilde{\omega})(0, r_{0}, \theta_{0}, \rho_{0}, \omega_{0})
= (r_{0}, \theta_{0}, \rho_{0}, \omega_{0}),
\end{gather*}
for $(r_{0}, \theta_{0}, \rho_{0}, \omega_{0} ) \in \Gamma_{R, U, J, Q}^{\  +. \tilde{\epsilon}}$,
where we took $R > 0$ such that
\begin{gather*}
\{ ((\del_{\rho}\psi )(\rho_{0}, \theta_{0} ) , \theta_{0}, \rho_{0}, -(\del_{\theta}\psi )(\rho_{0}, \theta_{0} ) )  :  (\theta_{0}, \rho_{0} ) \in U \times J \}
\subset
\Gamma_{R, U, J, Q}^{\  +. \tilde{\epsilon}}.
\end{gather*}
Set
\begin{gather*}
(r, \theta, \rho, \omega )(t; \rho_{0}, \theta_{0} )
:= (\tilde{r}, \tilde{\theta}, \tilde{\rho}, \tilde{\omega})(t,  (\del_{\rho}\psi )(\rho_{0}, \theta_{0} ) , \theta_{0}, \rho_{0}, -(\del_{\theta}\psi )(\rho_{0}, \theta_{0} )).
\end{gather*}
and consider the  map
\begin{gather}
(\rho_{0}, \theta_{0} ) \mapsto (\rho, \theta)(t; \rho_{0}, \theta_{0} ) \label{rho_theta_bijection}
\end{gather}
and its first derivatives.
We set 
$
\Theta := \{ (t, (\rho, \theta)(t; \rho_{0}, \theta_{0} ) ) | (\rho_{0}, \theta_{0}) \in J \times U \}.
$
By a straight computation, we obtain
\begin{gather*}
\Big| \frac{\del (\rho, \theta)(t, \rho_{0}, \theta_{0})}{\del(\rho_{0}, \theta_{0})} -
\left(
\begin{array}{cc}
1 & 0\\
0 & 1
\end{array}
\right)
\Big|
=
C\langle s \rangle^{- \tilde{\epsilon}}.
\end{gather*}
where $C$ depends on $C_{j, \alpha}$ and not on the choice of $\psi$ as long as $\psi$ satisfies  $ \eqref{psi_condition}$ for some $s$.
We fix a large enough $T > 0$ 
so that for any $ s > T$ we have
\begin{gather*}
\left|
\frac{\del (\rho, \theta)(t, \rho_{0}, \theta_{0})}{\del(\rho_{0}, \theta_{0})} -
\left(
\begin{array}{cc}
1 & 0\\
0 & 1
\end{array}
\right)
\right|
\ll 1.
\end{gather*}
Now  \eqref{rho_theta_bijection} becomes an injective map for every $t > 0$.
We denote its inverse by
\begin{gather*}
(\rho, \theta ) \mapsto(\rho_{0}, \theta_{0} )(t; \rho, \theta ).
\end{gather*}
Let
\begin{gather*}
Q(t; \rho_{0},  \theta_{0})\\
= \psi (\rho_{0}, \theta_{0})
+\int_{0}^{t}[ h( (r, \theta, \rho, \omega )(u; \rho_{0}\omega_{0})\\
+ \langle r(t; \rho_{0}, \theta_{0}), (\del_{u}\rho )(u; \rho, \theta ) \rangle
- \langle \omega(t; \rho_{0}, \theta_{0}), (\del_{u}\theta )(u; \rho, \theta ) \rangle
] du.
\end{gather*}
Then the function
\begin{equation*}
S(t, \rho, \theta ) =  Q(t; (\rho_{0}, \theta_{0})(t; \rho, \theta ))
\end{equation*}
defined on $\Theta$
is the desired solution to the Hamilton-Jacobi equation (see, for example,  \cite{DG97} Appendix A.3).
Moreover
\begin{gather*}
(\del_{\rho} S)(t, \rho, \theta ) = r(t; \rho_{0}(t, \rho, \theta ), \theta_{0}(t, \rho, \theta ) ),\\
-(\del_{\theta} S)(t, \rho, \theta ) = \omega(t; \rho_{0}(t, \rho, \theta ), \theta_{0}(t, \rho, \theta ) ).
\end{gather*}

The derivatives of $S(t, \rho, \theta)$
\begin{gather*}
\del_{\rho}^{j} \del_{\theta}^{\alpha} \del_{t} ( S(t, \rho , \theta ) - \frac{1}{2}t \rho^{2})
=\del_{\rho}^{j} \del_{\theta}^{\alpha} \tilde{h}( \del_{\rho}S(t, \rho , \theta ), \theta, \rho, -\del_{\theta}S(t, \rho , \theta ))
\end{gather*}
is a summation of the terms of the type
\begin{gather*}
(\del_{r}^{l} \del_{\theta}^{\beta} \del_{\rho}^{k} \del_{\omega}^{\gamma} \tilde{h})
( \del_{\rho}S(t, \rho , \theta ), \theta, \rho, -\del_{\theta}(t, \rho , \theta ))
\times\\
\prod_{i=1}^{l}\del_{\rho}^{k_{i}} \del_{\theta}^{\beta_{i}} (\del_{\rho} S)(t, \rho, \theta )
\times
\prod_{d=1}^{n-1} \prod_{j=1}^{\gamma_{d}} \del_{\rho}^{k_{d, j}} \del_{\theta}^{\beta_{d, j}} (-\del_{\theta_{d}} S)(t, \rho, \theta )),
\end{gather*}
which belongs to $\mathrm{O}(\langle t \rangle^{-m(l, \beta, k, \gamma) + l + (1- \tilde{\epsilon})|\gamma|})  \subset \mathrm{O}(\langle t \rangle^{-\epsilon})$.
This shows  \eqref{estimates_for_S}.

\end{proof}

Finally we extend $S(t, \rho, \theta)$ to a globally defined function on $\mathbb{R} \times (0, \infty) \times \del M$ which satisfies the same kind of estimates locally.

\begin{Thm}\label{SThmGlobal}
Let $h, \tilde{h}$ be as in Lemma  \ref{HtoF} defined on $T^{*}\mathbb{R} \times T^{*}\del M$.
Then there exists a function $S(t, \rho, \theta)$
defined on $T^{*}\mathbb{R} \times T^{*}\del M$
such that for every $J \Subset \mathbb{R} \setminus \{ 0 \}$, there exists $T > 0$ such that the Hamilton-Jacobi equation:
\begin{gather}
(\del_{t} S )(t, \rho, \theta ) = h( (\del_{\rho}S )(\rho, \theta ) , \theta, \rho, -(\del_{\theta}S )(\rho, \theta ) ) \label{S_Hamilton_Jacobi_eq}
\end{gather}
is satisfied for $t > |T|$, $\rho \in J$, and $\theta \in \del M$.
Moreover the function $S$ satisfies the following estimates:
\begin{gather}
|\del_{\rho}^{j} \del_{\theta}^{\alpha} (S(t, \rho, \theta ) - \frac{1}{2}t \rho^{2} )|
\leq \tilde{C}_{j, \alpha} \langle t \rangle^{1- \epsilon}. \label{estimates_for_SGlobal}
\end{gather}
\end{Thm}

\begin{proof}
First note that since $\del M$ is compact and the Hamilton-Jacobi equation is defined in a coordinate invariant manner, we can extend $U$ in the Theorem  \ref{SThm} to $\del M$.
It is sufficient to consider the case $J \Subset (0, \infty)$ and $t > T$, since we can extend the function $S$ in a $C^{\infty}$-fashion.

Take a sequence of open sets in $(0, \infty)$ such that
\begin{gather*}
J_{0} \Subset J_{1} \Subset J_{2} \Subset J_{3} \Subset \dots, \ \ 
\bigcup_{n = 0}^{\infty} J_{n} = (0, \infty ).
\end{gather*}

First we solve the Cauchy problem for the Hamilton-equation with initial data
\begin{equation*}
S(t, \rho , \theta) = \frac{1}{2}t \rho^{2} \ \ \ \text{when}\ \  \rho \in J_{1}, t = T_{1} >0
\end{equation*}
for a big enough $T_{1}$ by Theorem  \ref{SThm} with $U$ replaced by $\del M$.
We denote the solution by $S_{1}$.
We can assume that $S_{1}$ is defined on $(T_{1}, \infty) \times J_{0} \times \del M$.
$S_{1}$ also satisfies  \eqref{estimates_for_S} for $\rho \in J_{1}$ and $t  \geq T_{1}$.

Next  we take $\chi_{1} \in C_{0}^{\infty}(J_{1})$ equal to $1$ 
in a neighborhood of $\overline{J_{0}}$ (the closure of $J_{0}$).
We solve the Cauchy Problem
with initial  data
\begin{equation*}
S(t, \rho , \theta ) = \chi_{1}S_{1}+(1-\chi_{1})\frac{t}{2}\rho^{2}\ \ \ \text{when}\ \  \rho \in J_{2}, t=T_{2}.
\end{equation*}
By taking $T_{2} > T_{1}$ large enough,
the right hand side satisfies the conditions for $T_{2}$ in Theorem  \ref{SThm}.
So we can solve the Cauchy Problem for such $T_{2}$.
We denote the solution by $S_{2}$.

Repeating this procedure,
we obtain a sequence $S_{n}$ of functions and a sequence $T_{1} < T_{2} < \dots$ such that
$S_{n}$ is defined on $J_{n} \times [T_{n}, \infty ) \times \del M$,
\begin{equation*}
S_{n+1} = S_{m}\ \  \text{for} \ \ m \geq n+1 \ \ \text{on} \ \ J_{n} \times [T_{m}, \infty ) \times \del M,
\end{equation*}
and satisfies
 \eqref{estimates_for_S}.
Thus by extending in a $C^{\infty}$ fashion,
we can construct a $C^{\infty}$ function $S$
which satisfies  \eqref{S_Hamilton_Jacobi_eq},
and  \eqref{estimates_for_S}
for large enough $t$ and $\rho$ in  any fixed compact subset of $(0, \infty)$.

\end{proof}

\section{Proof of Theorem \ref{Wave_MainThm}}\label{Sec_Proof_of_the_main_Thm}

In this section, we give the proof of Theorem  \ref{Wave_MainThm}.
First we give the outline of the proof.

\begin{proof}[\bf{Outline of the Proof.}]

We consider the $t \to +\infty$ case.
By the density argument, it is sufficient to show the existence of the norm limit
\begin{equation*}
\lim_{t \to \infty} e^{itH} J e^{-iS(t, D_{r}, \theta )} u
\end{equation*}
for all $\hat{u} \in C_{0}^{\infty}( (\mathbb{R} \setminus \{ 0\} ) \times U_{\lambda})$ for all $\lambda$.
For such $u$, we have
\begin{gather*}
\frac{1}{i} e^{-itH}\frac{\del}{\del t}[e^{itH}Je^{-iS(t, D_{r}, \theta)} u] 
=[HJ-J\frac{\del S}{\del t}(t, D_{r}, \theta)] e^{-i S(t, D_{r}, \theta)} u.
\end{gather*}
By the Cook-Kuroda method we only need to show that
\begin{equation*}
\| [HJ-J\frac{\del S}{\del t}(t, D_{r}, \theta )] e^{-i S(t, D_{r}, \theta) } u \|_{\mathcal{H}} \in L_{t}^{1}(1, \infty).
\end{equation*}
We decompose 
\begin{gather*}
[HJ-J\frac{\del S}{\del t}(t, D_{r}, \theta )]\\
=[P_{0}J - JP_{0}] + V_{S}J + [V^{L}J - JV^{L}]
+J[P_{0} + V^{L}(r) - \frac{\del S}{\del t}(t, D_{r}, \theta)].
\end{gather*}
The fisrt three terms are essentially short range terms. It is easy to check
\begin{gather}
\| [P_{0}J - JP_{0}] + V_{S}J + [V^{L}J - JV^{L}]  e^{-i S(t, D_{r}, \theta )} u \|_{\mathcal{H}} \in L_{t}^{1}(1, \infty). \label{short_range_terms}
\end{gather}
We examine the last term:
\begin{gather*}
[P_{0} + V^{L}(r) - (\del_{t}S)(t, D_{r}, \theta)]  e^{-i S(t, D_{r}, \theta )} u\\
= h^{r}(r, \theta, D_{r}, -\frac{\del S}{\del \theta}(t, D_{r}, \theta ) ) e^{-iS(t, D_{r}, \theta ) } u\\
- h((\del_{\rho}S)(t, D_{r}, \theta ), \theta, D_{r}, -\frac{\del S}{\del \theta}(t, D_{r}, \theta ) ) e^{-iS(t, D_{r}, \theta ) } u\\
+  \bigl( \frac{1}{r} a_{2}^{j}D_{r} + \frac{1}{2 r^{2}} a_{3}^{j k}\frac{\del S}{\del \theta^{ k}}(t, D_{r}, \theta )  \bigr) e^{-iS(t, \rho, \theta ) ) } (\del_{\theta^{j}} u)\\
-\frac{1}{2 r^{2}} a_{3}^{j k} e^{-iS(t, \rho, \theta ) ) } (\del_{\theta^{j}} \del_{\theta^{k}} u)\\
+ \text{\ [ short range terms ]}.
\end{gather*}
We apply the stationary phase method to the first two terms.
Then the first terms which appear in the asymptotic expansion will vanish since the relation
\begin{gather*}
(\del_{\rho}S)(t, \rho, \theta ) = r
\end{gather*}
gives the stationary point with respect to the $\rho$-variable.
Therefore we obtain
\begin{gather}
\| [P_{0} + V^{L}(r) - (\del_{t}S)(t, D_{r}, \theta)]  e^{-i S(t, D_{r}, \theta )} u\|_{\mathcal{H}} \in L_{t}^{1}(1, \infty).  \label{long_range_terms}
\end{gather}
We give a detailed proof of  \eqref{short_range_terms} and  \eqref{long_range_terms} in the remaining of this section.
\end{proof}

First we consider the long-range term  \eqref{long_range_terms}.
The next proposition is our key estimate.
\begin{Prop}\label{LRProp}
Assume the assumptions of Theorem \ref{Wave_MainThm}. Suppose $u$ satisfies $\hat{u} \in C_{0}^{\infty}( (\mathbb{R} \setminus \{ 0\} ) \times U_{\lambda})$ and $J \times U$ is  a neighborhood of supp $\hat{u}$.
Then we have
\begin{gather}
|[\tilde{h}(r, \theta, D_{r}, -\frac{\del S}{\del \theta}(t, D_{r}, \theta))
 - \tilde{h}(\frac{\del S}{\del \rho}(t, D_{r}, \theta), \theta, D_{r}, -\frac{\del S}{\del \theta}(t, D_{r}, \theta))] \notag \\
\cdot e^{-iS(t, D_{r}, \theta)} u(r, \theta)| \ 
\leq C t^{-\frac{1}{2} - 1 - \epsilon } \label{VLin}
\end{gather}
for $(\frac{r}{t}, \theta) \in J \times U \Subset(0, \infty)\times \del M$, and
\begin{gather}
|[\tilde{h}(r, \theta, D_{r}, -\frac{\del S}{\del \theta}(t, D_{r}, \theta))
 - \tilde{h}(\frac{\del S}{\del \rho}(t, D_{r}, \theta), \theta, D_{r}, -\frac{\del S}{\del \theta}(t, D_{r}, \theta))]\notag \\
\cdot e^{-iS(t, D_{r}, \theta)} u(r, \theta)| \ 
\leq C_{N} (1+|r|+|t|)^{-N} \label{VLnotin}
\end{gather}
for any $N$ and for $(\frac{r}{t}, \theta) \notin J \times U$.
\end{Prop}

\begin{proof}[\bf{Proof of   \eqref{long_range_terms} } ]
We fix a neighborhood $J \times U$ of supp$\hat{u}$ which 
appears in Proposition \ref{LRProp}. Then
\begin{align*}
&\int_{1}^{\infty}
\|J|[\tilde{h}(r, \theta, D_{r}, -\frac{\del S}{\del \theta}(t, D_{r}, \theta))
 - \tilde{h}(\frac{\del S}{\del \rho}(t, D_{r}, \theta), \theta, D_{r}, -\frac{\del S}{\del \theta}(t, D_{r}, \theta))] \\
&e^{-iS(t, D_{r}, \theta)} u\| _{\mathcal{H}} dt\\
&=\int_{1}^{\infty}
\bigl(
\int_{\mathbb{R}_{+}\times \del M}|j(r)
[\tilde{h}(r, \theta, D_{r}, -\frac{\del S}{\del \theta}(t, D_{r}, \theta))\\
& - \tilde{h}(\frac{\del S}{\del \rho}(t, D_{r}, \theta), \theta, D_{r}, -\frac{\del S}{\del \theta}(t, D_{r}, \theta))] 
e^{-iS(t, D_{r}, \theta)} u(r, \theta)|^{2} dr H(\theta ) d\theta 
\bigr) ^{\frac{1}{2}}
dt\\
&\leq 
\int_{1}^{\infty}
\bigl(
\int_{\frac{r}{t} \in J}|j(r)
[\tilde{h}(r, \theta, D_{r}, -\frac{\del S}{\del \theta}(t, D_{r}, \theta))\\
& - \tilde{h}(\frac{\del S}{\del \rho}(t, D_{r}, \theta), \theta, D_{r}, -\frac{\del S}{\del \theta}(t, D_{r}, \theta))] 
e^{-iS(t, D_{r}, \theta)} u(r, \theta)|^{2} dr H(\theta ) d\theta 
\bigr) ^{\frac{1}{2}}\\
&+
\bigl(
\int_{\frac{r}{t} \notin J }|j(r)
[\tilde{h}(r, \theta, D_{r}, -\frac{\del S}{\del \theta}(t, D_{r}, \theta))\\
& - \tilde{h}(\frac{\del S}{\del \rho}(t, D_{r}, \theta), \theta, D_{r}, -\frac{\del S}{\del \theta}(t, D_{r}, \theta))] 
e^{-iS(t, D_{r}, \theta)} u(r, \theta)|^{2} dr H(\theta ) d\theta 
\bigr) ^{\frac{1}{2}}
dt .
\end{align*}
By  \eqref{VLin}, the first term is finite:
\begin{gather*}
\int_{1}^{\infty}
\bigl(
\int_{\frac{r}{t} \in J }|j(r)
[[\tilde{h}(r, \theta, D_{r}, -\frac{\del S}{\del \theta}(t, D_{r}, \theta))
 - \tilde{h}(\frac{\del S}{\del \rho}(t, D_{r}, \theta), \theta, D_{r}, -\frac{\del S}{\del \theta}(t, D_{r}, \theta))] \\
e^{-iS(t, D_{r}, \theta)} u(r, \theta)|^{2} dr H(\theta ) d\theta 
\bigr) ^{\frac{1}{2}}
dt\\
\leq
\int_{1}^{\infty}
\bigl(
\int_{R \in J }
|C t^{-\frac{1}{2}-1-\epsilon }|^{2}
t dR
\bigr) ^{\frac{1}{2}}
dt \ \ 
\leq
C \int_{1}^{\infty}t^{-1-\epsilon} dt\ \ 
< \infty .
\end{gather*}
By  \eqref{VLnotin}, the second term is also finite:
\begin{gather*}
\int_{1}^{\infty}
\bigl(
\int_{\frac{r}{t} \notin J }|j(r)
[[\tilde{h}(r, \theta, D_{r}, -\frac{\del S}{\del \theta}(t, D_{r}, \theta))
 - \tilde{h}(\frac{\del S}{\del \rho}(t, D_{r}, \theta), \theta, D_{r}, -\frac{\del S}{\del \theta}(t, D_{r}, \theta))] \\
e^{-iS(t, D_{r}, \theta)} u(r, \theta)|^{2} dr H(\theta ) d\theta 
\bigr) ^{\frac{1}{2}}
dt\\
\leq
\int_{1}^{\infty}
\bigl(
\int_{\frac{r}{t} \notin J }
C(1+|r|+|t|)^{-N}
dr 
\bigr) ^{\frac{1}{2}}
dt\ \ 
< \infty 
\end{gather*}
Therefore
\begin{gather*}
\|J[V^{L}(r, D_{r}, \theta) - V^{L}(\frac{\del W}{\del \rho}(D_{r}, \theta, t), D_{r}, \theta )]
e^{-iW(D_{r}, \theta, t)} u\| _{\mathcal{H}} \in L_{t}^{1}(1, \infty ).
\end{gather*}
\end{proof}

In order to prove Proposition \ref{LRProp}, we prepare a lemma.

\begin{Lem}\label{psiLem}
Let $S(t, \rho, \theta)$ satisfy the properties listed in Theorem  \ref{SThmGlobal}.
Set
\begin{equation*}
f_{r, \theta, t}(\rho ):= \frac{1}{t}(r\rho - S(t, \rho, \theta)).
\end{equation*}
For $\rho$ in any fixed compact subset of $\mathbb{R} \setminus \{ 0 \}$ and for  large enough $|t|$,
there exists a function $\Xi_{\theta , t}(r)$ 
which gives the critical point of $f_{r, \theta , t}(\rho )$:
\begin{equation*}
\del_{\rho}f_{r, \theta , t}(\rho ) = 0 
\iff
\rho = \Xi_{\theta , t}(r).
\end{equation*}
Set $\Omega_{d}:=(-d, d) $.
Then there exist $0 < \tilde{d} < d$ and a function $\phi_{r, \theta , t} \in C^{\infty}(\Omega_{d}; \mathbb{R})$
such that
$\Omega_{2\tilde{d}} \Subset \phi_{r, \theta ,t}(\Omega_{d})$.
Setting
\begin{gather*}
\psi_{r, \theta, t}(y) := \Xi_{\theta , t}(r) + \phi_{r, \theta ,t}(y),\\
(f_{r, \theta , t} \circ \psi_{r, \theta, t})(y) = 
f_{r, \theta , t} (\Xi_{\theta , t}(r)) + \langle A_{r, \theta , t}y, y\rangle /2,
\end{gather*}
where
\begin{equation*}
A_{r, \theta , t} = (\del_{\rho}^{2}f_{r, \theta , t})(\Xi_{\theta , t}(r)),
\end{equation*}
 we have
\begin{gather}
|\del_{y}^{k}\psi_{r, \theta , t}(0)| \leq C t^{-\epsilon} \ \ \ (k\geq 2),  \ \ 
\del_{y}\psi_{r, \theta , t}(0) = 1.\label{psi}\end{gather}
\end{Lem}
\begin{proof}

We only consider the $t  > 0$ and $ \rho > 0$ case.
First we prove that $\Xi_{\theta, t}(r)$ is well-defined. Compute
\begin{gather*}
0 = \del_{\rho }f_{r, \theta , t}(\rho ) \ \ 
  = \frac{1}{t} \big[ r - \frac{\del S}{\del \rho }(t, \rho , \theta ) \big]
\end{gather*}
We note that by \eqref{estimates_for_SGlobal},
\begin{align*}
|\frac{1}{t} \frac{\del^{2} S}{\del \rho ^{2}}(t, \rho , \theta) - 1| \leq C t^{-\epsilon }.
\end{align*}
This implies that $\frac{1}{t} \frac{\del S}{\del \rho}$ is monotonously increasing 
with respect to
$\rho$ for large enough $t$.
Thus there is a unique inverse function $\Xi_{\theta , t}(r)$ such that 
\begin{equation*}
(\del_{\rho}f_{r, \theta , t})(\Xi_{\theta , t}(r)) =0
\end{equation*}
for large enough $t$ and $\frac{r}{t} \in J$, a  fixed  compact subset of $(0, \infty)$.

Now we construct $\phi_{r, \theta , t}$ and $\psi_{r, \theta , t}$.
We set
\begin{gather*}
A_{r, \theta, t}
:= f_{r, \theta , t}^{\prime \prime}(\Xi_{\theta , t}(r)) \ 
= - \frac{1}{t}\frac{\del^{2} S}{\del \rho ^{2}}(t, \Xi_{\theta , t}(r), \theta).
\end{gather*}
Then  \eqref{estimates_for_SGlobal} implies that
\begin{equation*}
|A_{r, \theta , t}+1| \leq C t^{-\epsilon}.
\end{equation*}
Hence we have $A_{r, \theta , t} \to -1$ uniformly for $r/t \in J$.
If we set
\begin{equation*}
g_{r, \theta , t}(\rho ):= f_{r \theta , t}(\Xi_{\theta , t}(r) + \rho ),
\end{equation*}
then 
\begin{gather*}
g_{r, \theta , t}^{\prime}(0 ):= f_{r \theta , t}^{\prime}(\Xi_{\theta , t}(r)) = 0, \ \ 
g_{r, \theta , t}^{\prime \prime}(0 ):= f_{r \theta , t}^{\prime \prime}(\Xi_{\theta , t}(r)) = A_{r, \theta, t}, \\
g_{r, \theta , t}(\rho )-g_{r, \theta , t}(0) = \langle B_{r, \theta , t}(\rho ) \rho , \rho \rangle /2,
\end{gather*}
where
\begin{gather*}
B_{r, \theta , t}(\rho ) := 2\int_{0}^{1}g_{r, \theta , t}(s\rho )(1-s)ds, \ \ 
B_{r, \theta , t}(0) = A_{r, \theta , t}
\end{gather*}
by Taylor's formula.
Now we compute
\begin{align*}
|B_{r, \theta , t}(\rho ) - A_{r, \theta , t}|\ 
&= |B_{r, \theta , t}(\rho ) - B_{r, \theta , t}(0)|\ 
= 2|\int_{0}^{1}(g_{r, \theta , t}^{\prime \prime}(s\rho ) - g_{r, \theta , t}^{\prime \prime }(0) )(1-s) ds|\\
&\leq 2 \sup_{0 \leq s \leq 1} |g_{r, \theta , t}^{\prime \prime}(s\rho ) - g_{r, \theta , t}^{\prime \prime}(0)|\\
&\leq 2 \sup_{0 \leq s \leq 1} 
|\frac{1}{t} \frac{\del ^{2} S}{\del \rho ^{2} } (t, \Xi _{\theta , t}(r) + s \rho , \theta)
-\frac{1}{t} \frac{\del ^{2} S}{\del \rho ^{2} } (t, \Xi _{\theta , t}(r), \theta)| \\
&\leq C t^{-\epsilon}
\to 0
\end{align*}
as $t \to \infty$, uniformly for $\frac{r}{t}, \rho \in J$ by  \eqref{estimates_for_SGlobal}.
Hence by taking $t$ sufficiently large, we may assume
$\big| \frac{B_{r, \theta , t}(\rho )}{A_{r, \theta , t}(\rho )} - 1 \big| < 1/2$ is uniformly very small.
For such $t, \frac{r}{t}, $and $\rho$, we set
\begin{equation*}
X_{r, \theta , t}(\rho ) := \sqrt{\frac{B_{r, \theta , t}(\rho )}{A_{r, \theta , t}}} \cdot \rho.
\end{equation*}
Then we have
\begin{equation*}
g_{r, \theta , t}(\rho) - g_{r, \theta ,t}(0) 
= \langle A_{r, \theta , t} X_{r, \theta , t}(\rho ), X_{r, \theta , t}(\rho )\rangle/2.
\end{equation*}
Now we compute
\begin{align*}
(\del_{\rho }X_{t, \theta , t})(\rho )
&= \big( \sqrt{\frac{B_{r, \theta , t}(\rho )}{A_{r, \theta , t}}} \big) ^{\prime } \cdot \rho
+\sqrt{\frac{B_{r, \theta , t}(\rho )}{A_{r, \theta , t}}} \cdot 1\\
&= \frac{1}{\sqrt{A_{r, \theta , t}}}2\sqrt{B_{r, \theta , t}(\rho )}
\cdot B_{r, \theta , t}^{\prime}(\rho )\cdot \rho
+ \sqrt{\frac{B_{r, \theta , t}(\rho )}{A_{r, \theta , t}}} \cdot 1,\\
(\del_{\rho}B_{r, \theta , t})(\rho ) 
&= 2 \int_{0}^{1}g_{r, \theta , t}^{\prime \prime \prime}(s \rho )s(1-s) ds,\\
|g_{r, \theta, t}^{\prime \prime \prime}(s \rho )| 
&= |-\frac{1}{t}(\del_{\rho}^{3}S)(t, \Xi_{\theta , t}(r) + s\rho , \theta)|\leq C t^{-\epsilon},\\
|\del_{\rho} X_{r, \theta , t}(\rho) - 1| &\leq C t^{-\epsilon}.
\end{align*}
Hence for small enough $d_{0} >0$ and for $|\rho|  \leq d_{0}$,
we have
$|\del_{\rho} X_{r, \theta , t}(\rho) - 1|$
arbitrary small for all large enough $t$,
and
$
X_{r, \theta , t} : \Omega_{d_{0}} \to X_{r, \theta , t} (\Omega_{d_{0}})
$
is a $C^{\infty}$-diffeomorphism.
We can pick $d>0$ such that,
$
\Omega_{d} \subset X_{r, \theta , t} (\Omega_{d_{0}})
$
for all $r, \theta ,$ large enough $t, \frac{r}{t} \in J$.
Let $\phi_{r, \theta , t}$ be the inverse map of $X_{r, \theta , t}$
with domain $\Omega_{d}$.
Then we can also pick $\tilde{d} >  0$ such that
$\Omega_{\tilde{d}} \subset \phi_{r, \theta , t}(\Omega_{d})$
for all $r, \theta,$ large enough $t, \frac{r}{t} \in J$. 
We note that
\begin{align*}
g_{r, \theta , t} \circ \phi_{r, \theta , t} (y) -g_{r, \theta , t}(0)
&= \langle A_{r, \theta , t} X_{r, \theta , t} \circ \phi_{r, \theta , t}(y ),
X_{r, \theta , t} \circ \phi_{r, \theta , t}(y ) \rangle /2.\\
&= \langle A_{r, \theta , t} y, y\rangle/2.
\end{align*}
We set
\begin{equation*}
\psi_{r, \theta ,t}(y) = \phi_{r, \theta , t}(y) +\Xi_{\theta ,t}(r).
\end{equation*}
Then we have
\begin{equation*}
f_{r, \theta , t}\circ \psi_{r, \theta , t}(y)
= f_{r, \theta ,t}(\Xi_{\theta ,t}(r)) + \langle A_{r, \theta , t} y, y \rangle/2.
\end{equation*}

Lastly, we prove the estimates $ \eqref{psi}$.
For $k \geq 1$,
\begin{align*}
|\del_{\rho}^{k}B_{r, \theta , t}(\rho)|
&= 2|\int_{0}^{1}g_{r, \theta , t}^{2+k}(s \rho ) s^{k}(1-s) ds |
\leq 2 \sup |g_{r, \theta , t}^{2+k}(s\rho) |\\
&\leq 2 \sup |\frac{1}{t}\del_{\rho}^{2+k}S(t, \Xi_{\theta , t}(r) +s\rho, \theta)|
\leq C t^{-\epsilon}
\end{align*}
by  \eqref{estimates_for_SGlobal}.
We also have 
\begin{equation*}
|\del_{\rho}^{k}\sqrt{B_{r, \theta , t}}(\rho)| \leq C t^{-\epsilon}.
\end{equation*}
Therefore
\begin{gather*}
|\del_{\rho}^{k}X_{r, \theta ,t}(\rho )| \leq C t^{-\epsilon} \ \ \ \ (k\geq 2) ,\ \ 
|\del_{\rho}X_{r, \theta , t}(0) - 1|= C t^{-\epsilon},
\end{gather*}
and we have
\begin{gather*}
|\del_{y}^{k} \psi_{r, \theta , t}(y)|=|\del_{y}^{k}\phi_{r, \theta , t}(y)| 
\leq C t^{-\epsilon}\ \ \ \  (k\geq 2), \\
|\del_{y}\psi_{r, \theta , t}(0) - 1| = |\del_{y}\phi_{r, \theta , t}(0) - 1| \leq C t^{- \epsilon}.
\end{gather*}
Then we complete the proof of Lemma  \ref{psiLem}.

\end{proof}

\begin{proof}[\bf{Proof of Proposition  \ref{LRProp}.} ]

First we prove  \eqref{VLin}
for $\frac{r}{t} \in J$.
We fix $\chi \in C_{0}^{\infty}(\mathbb{R})$
such that
$\chi(x) = 1 $ if $|x| \leq \frac{1}{2}$, and
$\chi(x) = 0 $ if $|x| \geq 1$.
We split $u$ into two terms depending on $r, \theta ,$ and $ t$:
\begin{gather*}
\hat{u}_{r, \theta , t}^{c}(\rho , \theta )
= \hat{u}(\rho , \theta ) \chi \big( \frac{\rho - \Xi_{\theta , t}(r)}{\tilde{d}} \big),\\
\hat{u}_{r, \theta , t}^{d}(\rho , \theta )
= \hat{u}(\rho , \theta ) \big[ 1 - \chi \big( \frac{\rho - \Xi_{\theta , t}(r)}{\tilde{d}} \big) \big]
\end{gather*}
where we use notations defined in Lemma  \ref{psiLem}.
The support of $\hat{u}_{r, \theta , t}^{c}$ is  close to the critical point of
$r\rho - S(t, \rho, \theta)$, while
that of $\hat{u}^{d}_{r, \theta, t}$ is apart from it.
Note that
\begin{equation*}
\text{supp} \hat{u}_{r, \theta , t}^{c}
\subset \Xi_{\theta, t}(r) + \Omega_{\tilde{d}}
\subset \text{Ran}(\psi_{r, \theta , t}).
\end{equation*}
By a change of variables we have
\begin{align*}
 &\tilde{h}(\frac{\del S}{\del \rho}(t, D_{r}, \theta), \theta, D_{r}, -\frac{\del S}{\del \theta}(t, D_{r}, \theta))
e^{-iS(t, D_{r}, \theta)} u_{r, \theta , t}^{c}(r, \theta) \notag \\
= &\frac{1}{2\pi}
\int_{\mathbb{R}} 
\tilde{h}(\frac{\del S}{\del \rho}(t, \rho, \theta), \theta, \rho, -\frac{\del S}{\del \theta}(t, \rho, \theta))
e^{ir \rho -iS(t, \rho, \theta)} \hat{u}_{r, \theta , t}^{c}(\rho, \theta )d\rho\\
= &\frac{1}{2\pi}
\int_{\Omega_{d}}
\tilde{h}(\frac{\del S}{\del \rho}(t, \psi_{r, \theta , t}(y), \theta), \theta, \psi_{r, \theta , t}(y), -\frac{\del S}{\del \theta}(t, \psi_{r, \theta , t}(y), \theta))
\hat{u}_{r, \theta , t}^{c}(\psi_{r, \theta , t}(y), \theta )\\
 &\cdot J_{r, \theta, t}(y)
e^{itf_{r, \theta, t}(\Xi_{\theta, t}(r))} e^{it \langle A_{r, \theta , t}y, y\rangle /2} dy
\end{align*}
where $J_{r, \theta, t}(y) = |\psi_{r, \theta , t}^{\prime}(y)|$
is the Jacobian.
Since
\begin{gather*}
|D_{\rho}^{j}\tilde{h}(\frac{\del S}{\del \rho}(t, \rho, \theta), \theta, \rho, -\frac{\del S}{\del \theta}(t, \rho, \theta))
| \leq C t^{-\epsilon},
\end{gather*}
we have
\begin{gather*}
|D_{y}^{j}\tilde{h}(\frac{\del S}{\del \rho}(t, \psi_{r, \theta , t}(y), \theta), \theta, \psi_{r, \theta , t}(y), -\frac{\del S}{\del \theta}(t, \psi_{r, \theta , t}(y), \theta))
|
\leq C t^{-\epsilon},\\
|\tilde{h}(r, \theta, \psi_{r, \theta , t}(y), -\frac{\del S}{\del \theta}(t, \psi_{r, \theta , t}(y), \theta))
|
\leq C t^{-\epsilon},\\
|D_{y}^{j}\hat{u}_{r, \theta , t}^{c}(\psi_{r, \theta , t}(y), \theta )|
\leq C,\ \\
|D_{y}^{j}J_{r, \theta, t}(y)|
\leq C,
\end{gather*}
for $y \in \Omega_{d}$, $\frac{r}{t} \in J$.
Now we apply the stationary phase method (see H\"{o}rmander~\cite{Ho83-85} Section 7.7) to the integral.
In the asymptotic expansion of
\begin{align*}
&[\tilde{h}(r, \theta, D_{r}, -\frac{\del S}{\del \theta}(t, D_{r}, \theta))
- \tilde{h}(\frac{\del S}{\del \rho}(t, D_{r}, \theta), \theta, D_{r}, -\frac{\del S}{\del \theta}(t, D_{r}, \theta))
]\\
&\cdot e^{-iS(t, D_{r}, \theta)} u_{r, \theta , t}^{c}(r, \theta)\\
=& \frac{1}{2\pi}
\int_{\Omega_{d}}
[\tilde{h}(r, \theta, \psi_{r, \theta , t}(y), -\frac{\del S}{\del \theta}(t, \psi_{r, \theta , t}(y), \theta)) \\
&- \tilde{h}(\frac{\del S}{\del \rho}(t, \psi_{r, \theta , t}(y), \theta), \theta, \psi_{r, \theta , t}(y), -\frac{\del S}{\del \theta}(t, \psi_{r, \theta , t}(y), \theta))
] \cdot \\
&\hat{u}_{r, \theta , t}^{c}(\psi_{r, \theta , t}(y), \theta ) \cdot J_{r, \theta, t}(y)
e^{itf_{r, \theta, t}(\Xi_{\theta, t}(r))} e^{it \langle A_{r, \theta , t}y, y \rangle /2} dy,
\end{align*}
the terms in which $\tilde{h}$ is not differentiated will vanish
since
\begin{equation*}
\frac{\del S}{\del \rho}(t, \psi_{r, \theta, t}(0), \theta) = r.
\end{equation*}
Especially, in the first step of the asymptotic expansion,
we need to estimate only the remainder terms.
Therefore we have
\begin{align*}
&\big(
[\tilde{h}(r, \theta, D_{r}, -\frac{\del S}{\del \theta}(t, D_{r}, \theta))
- \tilde{h}(\frac{\del S}{\del \rho}(t, D_{r}, \theta), \theta, D_{r}, -\frac{\del S}{\del \theta}(t, D_{r}, \theta))
]\\
&\cdot e^{-iS(t, D_{r}, \theta)} u_{r, \theta , t}^{c}(r, \theta)\\
&\leq
C t^{-\frac{1}{2} - 1}
\sum_{|k|\leq 3} \sup ||D_{y}^{k}
[\tilde{h}(r, \theta, \psi_{r, \theta , t}(y), -\frac{\del S}{\del \theta}(t, \psi_{r, \theta , t}(y), \theta)) \\
&- \tilde{h}(\frac{\del S}{\del \rho}(t, \psi_{r, \theta , t}(y), \theta), \theta, \psi_{r, \theta , t}(y), -\frac{\del S}{\del \theta}(t, \psi_{r, \theta , t}(y), \theta))
] \cdot \\
&\hat{u}_{r, \theta , t}^{c}(\psi_{r, \theta , t}(y), \theta ) \cdot J_{r, \theta, t}(y)
||_{L^{2}}\\
&\leq C t^{ - \frac{1}{2} - 1 - \epsilon}.
\end{align*}

We now consider $u_{r, \theta, t}^{d}$  term.
\begin{align*}
 &\tilde{h}(r, \theta, D_{r}, -\frac{\del S}{\del \theta}(t, D_{r}, \theta))
e^{-iS(t, D_{r}, \theta)} u_{r, \theta , t}^{d}(r, \theta) \\
= &\frac{1}{2\pi}
\big(
\int_{-\infty} ^{\Xi_{\theta, t}(r) - \frac{1}{2}\tilde{d}}
+\int_{\Xi_{\theta, t}(r) + \frac{1}{2}\tilde{d}}^{\infty}   \big)  \\
& \tilde{h}(\frac{\del S}{\del \rho}(t, \rho, \theta), \theta, \rho, -\frac{\del S}{\del \theta}(t, \rho, \theta))
e^{ir\rho -iS(t, \rho, \theta)} \hat{u}_{r, \theta , t}^{d}(\rho, \theta )d\rho
\end{align*}
We consider integration over $\geq \Xi_{\theta, t}(r) + \frac{1}{2}\tilde{d}$ only
(the other part is similar to prove).
 \eqref{estimates_for_SGlobal} implies 
\begin{gather*}
\del_{\rho} f_{r, \theta, t}(\rho) \leq - C < 0,\\
|\del_{\rho}^{j} f_{r, \theta, t}(\rho)| \leq  C t^{-\epsilon} , \  j \geq 2
\end{gather*}
in this region.
Let $y \mapsto h_{r, \theta,t}(y)$ be the inverse of $\rho \mapsto f_{r, \theta, t}(\rho)$.
Then
\begin{gather*}
|\del_{y} h_{r, \theta, t}(y)| \leq C,\\
|\del_{y}^{j} h_{r, \theta, t}(y)| \leq  C t^{-\epsilon} , \ j \geq 2.
\end{gather*}
By a change of the variables we obtain
\begin{align*}
&|\int_{\Xi_{\theta, t}(r) + \frac{1}{2}\tilde{d}}^{\infty}
\tilde{h}(\frac{\del S}{\del \rho}(t, \rho, \theta), \theta, \rho, -\frac{\del S}{\del \theta}(t, \rho, \theta))
e^{it f_{r, \theta, t}(\rho)} \hat{u}_{r, \theta , t}^{d}(\rho, \theta )d\rho | \\
&=
\big|\int e^{it y} 
\tilde{h}(\frac{\del S}{\del \rho}(t, h_{r, \theta, t}(y), \theta), \theta, h_{r, \theta, t}(y), -\frac{\del S}{\del \theta}(t, h_{r, \theta, t}(y), \theta))
\hat{u}_{r, \theta , t}^{d}(h_{r, \theta, t}(y), \theta )\cdot\\
&|h_{r, \theta, t}^{\prime}(y) \big|
dy | \\
&\leq C t^{-N}.
\end{align*}
We can show the same kind of estimations for\\
$\tilde{h}(r, \theta, D_{r}, -\frac{\del S}{\del \theta}(t, D_{r}, \theta))e^{-iS(t, D_{r}, \theta)} u_{r, \theta , t}^{d}(r, \theta)$.
We have ended the proof of  \eqref{VLin}.

Now we show  \eqref{VLnotin}.
 \eqref{estimates_for_SGlobal} implies that there exists $\tilde{J}$ such that
$\frac{1}{t}\frac{\del S}{\del \rho} \in  \tilde{J} \Subset J$
for large enough $t$.
Thus the absolute value of the derivative of\\ $\rho \mapsto (r\rho - S(t, \rho, \theta))/(|r|+|t|)$
is bounded below  for $\frac{r}{t} \notin J$,
large enough $t$, and $\rho \in$ supp$\hat{u}$.
Thus we can apply the non-stationary phase method and obtain  \eqref{VLnotin}.
\end{proof}

\begin{proof}[\bf{Proof of  \eqref{short_range_terms}}, partial isometry, and intertwining property.]

First we consider the\\
short-range terms.
On $\tilde{M_{0}}$, 
\begin{align*}
P_{0}J - JP_{0} + V^{L}J-JV^{L}
&=O(r^{-\frac{n-1}{2}}r^{-1-\epsilon})\del_{r} \del_{\theta }
+O(r^{-\frac{n-1}{2}}r^{-2})\del_{\theta}^2\\
&+\sum_{j} O(r^{-\frac{n-1}{2}}r^{-1-\epsilon})\del_{r}^j.
\end{align*}
These terms can be treated as a short range perturbation of $(1-\epsilon) = \mu_{S}$ type.
Hence on  $\tilde{M_{0}}$,
$P_{0}J - JP_{0} + V^{L}J-JV^{L} +V^{S}J$ is a finite sum of terms of the form
$v_{j, \alpha}(r, \theta ) r^{-\frac{n-1}{2}} D_{r}^{j}\del_{\theta}^{\alpha}$
where $ v_{j, \alpha} $ satisfy
\begin{gather*}
\int_{\mathbb{R}_{+} \times U_{\lambda}}
|v_{j, \alpha}^{S}(x)|^{2}\langle x \rangle^{-M} G(x)dx	< \infty ,\\
\int_{1}^{\infty}(\int_{(\rho, \theta) \in J \times U}|v_{j, \alpha}^{S}(t\rho, \theta)|^{2}d\rho d\theta ) ^{1/2} 
t^{(1- \epsilon) |\alpha|}dt < \infty,
\end{gather*}
for some neighborhood $J \times U$
of almost every $(\rho_{0}, \theta_{0}) \in \mathbb{R}\times \partial M $.
We assume supp $\hat{u} \subset J \times U$.

We consider the differential operators with respect to $\theta$-variable.
$\del_{\theta}e^{-iS(t, D_{r}, \theta)}$ yields\\
 $(\del_{\theta}S)(t, D_{r}, \theta)$ terms
which increase as $t^{1-\epsilon}$.
Hence, as in the long-range case, the inequalities  \eqref{estimates_for_SGlobal} implies
\begin{gather}
|D_{r}^{j}\del_{\theta}^{\alpha} e^{-iS(t, D_{r}, \theta)} u(r, \theta )|   
=
\frac{1}{2\pi}|\int_{\mathbb{R}}
\del_{\theta}^{\alpha} [e^{i(r \rho  - S(t, \rho, \theta) )}
\rho^{j}\hat{u}(\rho , \theta )] d\rho |    
\leq C t^{-\frac{1}{2} + |\alpha|(1- \epsilon)} \label{SRin}
\end{gather}
for $(\frac{r}{t}, \theta) \in J \times U$,
and  
\begin{gather}
|\del_{r}^{j}\del_{\theta}^{\alpha} e^{-iS(t, D_{r}, \theta)} u(r, \theta )|
\leq C_{N} (1+|r|+|t|)^{-N} \label{SRnotin}
\end{gather}
for any $N$ for $(\frac{r}{t}, \theta) \notin J \times U$.
Thus we obtain for such $v_{j, \alpha} $
\begin{equation*}
\| v_{j, \alpha}(r, \theta ) r^{-\frac{n-1}{2}} D_{r}^{j}\del_{\theta}^{\alpha}
e^{-iS(t, D_{r}, \theta)} u\| _{\mathcal{H}}
\in L_{t}^{1}(1, \infty ),
\end{equation*}
which proves
\begin{equation*}
\|\big( [P_{0}J - JP_{f}]
+V^{S}J +[V^{L}J - JV^{L}] \big)
e^{-iS(t, D_{r}, \theta)} u\| _{\mathcal{H}}
\in L_{t}^{1}(1, \infty ).
\end{equation*}
We have proved the existence of the modified wave operators.

 \eqref{VLin}, \eqref{VLnotin}, \eqref{SRin}, and  \eqref{SRnotin} also show that 
$W_{\pm}$ are partial isometries from $\mathcal{H}_{f}$ into $\mathcal{H}$.

The intertwining property follows from
\begin{equation}
\s-lim_{t \to \infty} (e^{-iS(t+s, D_{r}, \theta)} - e^{-is(D_{r}, \theta, t)} e^{-isP_{f}}) = 0
\end{equation}
which can be proved using   \eqref{estimates_for_SGlobal} and the dominated convergence theorem.
The proof of the theorem is complete.
\end{proof}

\end{document}